\documentclass[10pt,twoside]{article}
\usepackage[latin1]{inputenc}
\usepackage[active]{srcltx}
\usepackage{amsfonts,enumerate}
\usepackage[mathscr]{euscript}
\usepackage{epsfig}
\usepackage{latexsym}
\usepackage[all]{xy}
\usepackage{amssymb}
\usepackage{amsthm}
\usepackage{amsmath}
\usepackage{amsxtra}
\usepackage{calrsfs}
\usepackage{url}

\theoremstyle{plain}
\newtheorem{prop}{Proposition}[section]
\newtheorem{thm}[prop]{Theorem}
\newtheorem{coro}[prop]{Corollary}

\newtheorem{lemma}[prop]{Lemma}

\newtheorem*{thm*}{Theorem}

\theoremstyle{definition}
\newtheorem*{defi}{Definition}

\theoremstyle{remark}

\newtheorem{remark}[prop]{Remark}

\numberwithin{table}{section}

\DeclareMathOperator{\Frob}{Frob}
\DeclareMathOperator{\Gal}{Gal}
\DeclareMathOperator{\Aut}{Aut}
\DeclareMathOperator{\End}{End}
\DeclareMathOperator{\Hom}{Hom}
\DeclareMathOperator{\Sch}{Sch}
\DeclareMathOperator{\Trace}{Tr}
\DeclareMathOperator{\Isom}{Isom}
\DeclareMathOperator{\Ind}{Ind}
\DeclareMathOperator{\HH}{H}

\DeclareMathOperator{\identity}{id}
\DeclareMathOperator{\ord}{ord}
\DeclareMathOperator{\Spec}{Spec}

\DeclareMathOperator{\sign}{sign}

\newcommand{\bfP}{\mathcal P}

\newcommand{\F}{\mathbb F}

\newcommand{\Om}{{\mathscr{O}}}

\newcommand{\Siegel}{\mathfrak{H}}
\newcommand{\GL}{{\rm GL}}
\newcommand{\SL}{{\rm SL}}

\def\ZZ{\mathbb Z}
\def\NN{\mathbb N}
\def\RR{\mathbb R}
\def\FF{\mathbb F}
\def\QQ{\mathbb Q}
\def\PP{\mathbb P}

\def\CC{\mathbb C}
\def\ff{\mathfrak f}

\def\hh{\mathfrak h}

\def\<#1>{{\left\langle{#1}\right\rangle}}

\def\Z{{\mathbb Z}}             
\def\Q{{\mathbb Q}}             
\def\R{{\mathbb R}}             
\def\q{{\mathfrak q}}  

\def\set#1{{\left\{{\def\st{\;:\;}#1}\right\}}}

\def\id#1{{\mathfrak{#1}}}      
\def\normid#1{{\norm{\id{#1}}}}
\DeclareMathOperator{\norm}{{\mathscr N}}

\DeclareMathOperator{\trace}{{\mathrm{Tr}}}

\def\YEAR{\year}\newcount\VOL\VOL=\YEAR\advance\VOL by-1995
\def\firstpage{1}\def\lastpage{1000}
\def\received{}\def\revised{}
\def\communicated{}

\makeatletter
\def\magnification{\afterassignment\m@g\count@}
\def\m@g{\mag=\count@\hsize6.5truein\vsize8.9truein\dimen\footins8truein}
\makeatother

\font\eightrm=cmr8
\font\caps=cmcsc10                    
\font\Caps=cmcsc10 scaled \magstep1   

\makeatletter
\@specialpagefalse\headheight=8.5pt
\def\DocMath{}
\renewcommand{\@evenhead}{%
    \ifnum\thepage>\lastpage\rlap{\thepage}\hfill%
    \else\rlap{\thepage}\slshape\leftmark\hfill{\caps\SAuthor}\hfill\fi}%
\renewcommand{\@oddhead}{%
    \ifnum\thepage=\firstpage{\DocMath\hfill\llap{\thepage}}%
    \else{\slshape\rightmark}\hfill{\caps\STitle}\hfill\llap{\thepage}\fi}%
\makeatother

\def\TSkip{\bigskip}
\newbox\TheTitle{\obeylines\gdef\GetTitle #1
\ShortTitle  #2
\SubTitle    #3
\Author      #4
\ShortAuthor #5
\EndTitle
{\setbox\TheTitle=\vbox{\baselineskip=20pt\let\par=\cr\obeylines%
\halign{\centerline{\Caps##}\cr\noalign{\medskip}\cr#1\cr}}%
	\copy\TheTitle\TSkip\TSkip%
\def\next{#2}\ifx\next\empty\gdef\STitle{#1}\else\gdef\STitle{#2}\fi%
\def\next{#3}\ifx\next\empty%
    \else\setbox\TheTitle=\vbox{\baselineskip=20pt\let\par=\cr\obeylines%
    \halign{\centerline{\caps##} #3\cr}}\copy\TheTitle\TSkip\TSkip\fi%
\centerline{\caps #4}\TSkip\TSkip%
\def\next{#5}\ifx\next\empty\gdef\SAuthor{#4}\else\gdef\SAuthor{#5}\fi%
\ifx\received\empty\relax
    \else\centerline{\eightrm Received: \received}\fi%
\ifx\revised\empty\TSkip%
    \else\centerline{\eightrm Revised: \revised}\TSkip\fi%
\ifx\communicated\empty\relax
    \else\centerline{\eightrm Communicated by \communicated}\fi\TSkip\TSkip%
\catcode'015=5}}\def\Title{\obeylines\GetTitle}
\def\Abstract{\begingroup\narrower
    \parskip=\medskipamount\parindent=0pt{\caps Abstract. }}
\def\EndAbstract{\par\endgroup\TSkip}

\long\def\MSC#1\EndMSC{\def\arg{#1}\ifx\arg\empty\relax\else
     {\par\narrower\noindent%
     2000 Mathematics Subject Classification: #1\par}\fi}

\long\def\KEY#1\EndKEY{\def\arg{#1}\ifx\arg\empty\relax\else
	{\par\narrower\noindent Keywords and Phrases: #1\par}\fi\TSkip}

\newbox\TheAdd\def\Addresses{\vfill\copy\TheAdd\vfill
    \ifodd\number\lastpage\vfill\eject\phantom{.}\vfill\eject\fi}
{\obeylines\gdef\GetAddress #1
\Address #2 
\Address #3
\Address #4
\EndAddress
{\def\xs{4.3truecm}\parindent=0pt
\setbox0=\vtop{{\obeylines\hsize=\xs#1\par}}\def\next{#2}
\ifx\next\empty 
     \setbox\TheAdd=\hbox to\hsize{\hfill\copy0\hfill}
\else\setbox1=\vtop{{\obeylines\hsize=\xs#2\par}}\def\next{#3}
\ifx\next\empty 
     \setbox\TheAdd=\hbox to\hsize{\hfill\copy0\hfill\copy1\hfill}
\else\setbox2=\vtop{{\obeylines\hsize=\xs#3\par}}\def\next{#4}
\ifx\next\empty\ 
     \setbox\TheAdd=\vtop{\hbox to\hsize{\hfill\copy0\hfill\copy1\hfill}
                \vskip20pt\hbox to\hsize{\hfill\copy2\hfill}}
\else\setbox3=\vtop{{\obeylines\hsize=\xs#4\par}}
     \setbox\TheAdd=\vtop{\hbox to\hsize{\hfill\copy0\hfill\copy1\hfill}
	        \vskip20pt\hbox to\hsize{\hfill\copy2\hfill\copy3\hfill}}
\fi\fi\fi\catcode'015=5}}\gdef\Address{\obeylines\GetAddress}

\begin{document}
\Title Modularity of the Consani-Scholten quintic
\ShortTitle Modularity of the Consani-Scholten quintic
\SubTitle   
with an Appendix by Jos\'e Burgos Gil and Ariel Pacetti
\Author 

Luis Dieulefait, Ariel Pacetti \footnote{\eightrm AP partially supported by CONICET PIP 2010-2012 and FonCyT BID-PICT 2010-0681.} and Matthias Sch\"utt \footnote{\eightrm MS partially supported by DFG under grant Schu 2266/2-2 and by ERC through StG~279723 (SURFARI).}$^,$
\footnote{\eightrm JBG partially supported by grants MTM2009-14163-c02-01 and CSIC-2009501001.} 

\ShortAuthor Luis Dieulefait, Ariel Pacetti and Matthias Sch\"utt 
\EndTitle

\Abstract 
We prove that the Consani-Scholten quintic, a Calabi-Yau threefold
over $\Q$, is Hilbert modular.  For this, we refine several techniques
known from the context of modular forms.  Most notably, we extend the
Faltings-Serre-Livn\'e method to induced four-dimensional Galois
representations over $\Q$.  We also need a Sturm bound for Hilbert
modular forms; this is developed in an appendix by Jos\'e Burgos Gil
and the second author.

\EndAbstract

\MSC Primary: 11F41; Secondary: 11F80, 11G40, 14G10, 14J32
\EndMSC
\KEY Consani-Scholten quintic, Hilbert modular form, Faltings--Serre--Livn\'e method, Sturm bound
\EndKEY
\Address 
Luis Dieulefait, 
Departament d'\`Algebra i Geometria,
Facultat de Matem\`atiques, 
Universitat de Barcelona, 
Gran Via de les Corts Catalanes, 585. 
08007 Barcelona.
ldieulefait@ub.edu

\Address 
Ariel Pacetti
Departamento de Matem\'atica, 
Universidad de Buenos Aires, 
Pabell\'on I, Ciudad Universitaria. C.P:1428, 
Buenos Aires, Argentina
apacetti@dm.uba.ar
\Address Matthias Sch\"ut
Institut f\"ur Algebraische Geometrie, 
Leibniz Universit\"at Hannover, 
Welfengarten 1, 30167 
Hannover, Germany
schuett@math.uni-hannover.de
\Address Jos\'e Burgos Gil
ICMAT-CSIC
jiburgosgil@gmail.com
\EndAddress

\section{Introduction}

The modularity conjecture for Calabi-Yau threefolds defined over $\Q$
is a particular instance of the Langlands correspondence. Given a
Calabi-Yau threefold $X$ over $\Q$ we consider the compatible family
of Galois representations $\rho_\ell$ of dimension, say, $n$, giving
the action of $G_\Q = \Gal(\bar{\Q}/\Q)$ on $\HH^3(X_{\bar{\Q}},
\Q_\ell)$: the conjecture says that there should exist an automorphic
form $\pi$ of $\GL_n$ such that $\ell$-adic Galois representations
attached to $\pi$ are isomorphic to the representations $\rho_\ell$.  This implies that the
$L$-functions of $\pi$ and $\rho_\ell$ agree, at least up to finitely
many local factors. Observe that (according to Langlands
functoriality) $\pi$ should be cuspidal if and only if the
representations $\rho_\ell$ are absolutely irreducible.\\ The only
case where the conjecture is known in general (among Calabi-Yau
threefolds) is the rigid case, i.e., the case with $n=2$.  In this
case modularity was established by the first author and Manoharmayum
(cf.~\cite{DM}) under some mild local conditions. These local
assumptions are no longer required since it is known that modularity
of all rigid Calabi-Yau threefolds defined over $\Q$ follows from
Serre's conjecture and the later has recently been proved
(cf. ~\cite{KW1}, ~\cite{KW2}, \cite{d-rigid}, \cite{gy}). \\ It was observed by Hulek and
Verrill in ~\cite{HV} that the modularity result in ~\cite{DM} can be
extended to show modularity of those Calabi-Yau threefolds such that
the representations $\rho_\ell$ are (for every $\ell$) reducible and
have $2$-dimensional irreducible components. In fact, using Serre's
conjecture, one can show that this is true even if reducibility occurs
only after extending scalars, assuming that reducibility is a uniform
property, i.e., independent of $\ell$ (this uniformity follows for
instance from Tate's conjecture).
 \newline

In this paper,
we will prove modularity for a non-rigid Calabi-Yau threefold over $\Q$
such that the  representations $\rho_\ell$ are 
irreducible.
To our knowledge such an example has not been known before.
We sketch the basic set-up:

In \cite{consani}, Consani and Scholten consider a
quintic threefold $\tilde X$ which we will review in section
\ref{s:geom}. It has good reduction outside the set $\{2,3,5\}$ and
Hodge numbers:
\[
h^{3,0}=1=h^{2,1}, \quad h^{2,0}=0=h^{1,0} \text{ and }h^{1,1}=141.
\]
In particular the third \'etale cohomology is four dimensional. If we
fix a prime $\ell$, the action of the Galois group
$\Gal(\bar{\Q}/\Q)$ on the third \'etale cohomology gives a
$4$-dimensional representation
\[
\rho_\ell:\Gal(\bar{\Q}/\Q) \to \GL(\HH^3(\tilde X_{\bar{\Q}},\Q_\ell))
\simeq \GL_4(\Q_\ell).
\]
Let $F = \Q[\sqrt{5}]$ and $\Om_F$ its ring of integers. In \cite{consani} it
is shown that the restriction
\[
\rho_\ell|_{\Gal(\bar{\Q}/F)}:\Gal(\bar{\Q}/F) \to
\Aut\left(\HH^3(\tilde{X}_{\bar{\Q}},\Q_\ell)\otimes \overline{\Q_\ell}\right),
\]
is the direct sum of two $2$-dimensional representations (see Theorem 3.2
of \cite{consani}). 
More precisely, if $\lambda$ is a prime of
$\Om_F$ over $\ell$, then there exists a $2$-dimensional
representation
\[
\sigma_{\lambda} : \Gal(\bar{\Q}/F) \to \GL_2(\Om_\lambda),
\]
such that $\rho_\ell|_{\Gal(\bar{\Q}/F)}$ is a direct sum of $\sigma_\lambda$ and
$\sigma_\lambda'$ (the external conjugate of $\sigma_\lambda$) and $\rho_\ell = \Ind^\Q_F \sigma_\lambda$.

In the same work, an holomorphic Hilbert newform $\ff$ on $F$ of
weight $(2,4)$ and level $\id{c}_\ff = (30)$ is constructed, whose
$L$-series is conjectured to agree with that of $\sigma$. The aim of
this work is to prove this modularity result.

Let $\lambda$ be a prime of $F$ over a rational prime $\ell$. Let
$\Om_\lambda$ denote the completion at $\lambda$ of $\Om_F$.  Since
$\ff$ has $F$-rational eigenvalues, by the work of Taylor (see
\cite{taylor}), there exists a two-dimensional continuous
$\lambda$-adic Galois representation
\[
\sigma_{\ff,\lambda}:\Gal(\bar{\QQ}/F) \to \GL_2(\Om_\lambda),
\]
with the following properties:
$\sigma_{\ff,\lambda}$ is unramified outside $\ell \id{c}_\ff$,
and if $\id{p}$
is a prime of $F$ not dividing $\ell \id{c}_\ff$, then
\begin{eqnarray*}
\trace \sigma_{\ff,\lambda}(\Frob \id{p}) &=& \theta(T_\id{p}),\\
\det \sigma_{\ff,\lambda}(\Frob \id{p}) &=& \theta(S_\id{p}) \normid{p}.
\end{eqnarray*}
Here $T_{\id{p}}$ denotes the $\id{p}$-th Hecke operator,
$S_{\id{p}}$ denotes the diamond operator (given by the action of the
matrix $\left(\begin{smallmatrix}\alpha
  &0\\ 0&\alpha \end{smallmatrix}\right)$, for $\alpha =
\prod_{\id{p}} \pi_\id{q}^{v_\id{q}(\id{p})}$ and $\pi_\id{q}$ a
local uniformizer) and 
$\theta(T)$ is the eigenvalue of the Hecke operator $T$ on $\ff$.
Let
$\tau(\ff)$ be the Hilbert modular form which is the external Galois conjugate
of $\ff$, where $\tau$ is the order two element in $\Gal(F/\QQ)$.
Observe that the $\lambda$-adic Galois representations attached to $\tau(\ff)$
are obtained by applying $\tau$ to the traces of the images of Frobenius in the Galois representations attached to $\ff$. Our
result can be stated as:


\begin{thm} The representations $\sigma_\lambda$ and $\sigma_{\nu(\ff),\lambda}$ are isomorphic, where $\nu$
is
 either the identity or $\tau$.
\label{main-theorem}
\end{thm}

In particular the theorem implies that the $L$-series of
$\sigma_\lambda$ and $\sigma_{\nu(\ff),\lambda}$ agree.  This
solves the conjecture from \cite{consani}.

\begin{remark}
By known cases of automorphic base change (theta lift) and
functoriality, Theorem \ref{main-theorem} is known to imply that $\rho_\ell$ corresponds to a
Siegel modular form of genus $2$ and to a cuspidal automorphic form
of
$\GL_4$.
\end{remark}

\begin{remark}
Dimitrov has proved a Modularity Lifting Theorem that
applies to Hilbert modular forms of non-parallel weight (cf. ~\cite{Di}),
and he and the first author checked  that for $\ell=7$ the
representation $\sigma$ satisfies all the technical conditions of
this theorem (cf. ~\cite{DD}).
Thus it would be enough to prove residual modularity
modulo $7$ to deduce the modularity of $\sigma$ from this result. We
will follow, however, a different path.
\end{remark}

\begin{remark}
In 2006, one of us received a preprint by Y.-C.~Yi
which claimed to prove the Hilbert modularity of $\tilde X$.
However, the arguments presented  contained several crucial gaps and inaccuracies,
and to our knowledge, the preprint has never been published properly.
\end{remark}

We give an outline of the proof of Theorem \ref{main-theorem}.  Since
both Galois representations come in compatible families, it is enough
to prove that they are isomorphic for a specific choice of primes
$\lambda$ over $\ell$. We choose $\ell=2$ so as to apply a
Faltings-Serre method of proving that two given $2$-adic Galois
representations are isomorphic (cf. \cite{Livne}). Actually, in
\cite{consani} it is proven that $\sigma_\lambda$ exists, but its trace at a prime is only determined up to conjugation by $\tau$.
Hence in Theorem~\ref{Livne} we give a version of
the Faltings-Serre method that applies to reducible $4$-dimensional representations
of $\Gal(\bar\QQ/F)$.

Theorem \ref{Livne} implies that the 2-dimensional Galois representations 
$\sigma_\lambda, \sigma_{\nu(\ff),\lambda}$
have isomorphic
semisimplifications.
Since $\ff$ is a cuspidal Hilbert eigenform,
its $\lambda$-adic Galois representations are irreducible for all
primes $\lambda$, and Livn\'e's Theorem asserts in particular that the
same is true for the representations attached to $\tilde X$.
Thus we deduce Theorem \ref{main-theorem}.

\smallskip

In the present situation, the  problem is  non-trivial notably because
$2$ is inert in $\Om_F$.  Hence the residual representations lie a
priori in $\GL_2(\F_4)$, and actually in $\SL_2(\F_4)$. This is
clear for both representations: the representation
$\sigma_{\ff,\lambda}$ has trivial nebentypus so the determinant image
lies in $\F_2^\times$ while the representation $\sigma_{\lambda}$ at a prime
ideal $\id{p}$ has real determinant of absolute value
$\normid{p}^3$.

The group $\SL_2(\F_4)$ is not a solvable group (it is in fact
isomorphic to $A_5$, the alternating group of $5$ elements).  We will
overcome this subtlety by showing that the images of the residual
representations are $2$-groups.  For $\sigma_{\ff,2}$, this will be
achieved in section \ref{s:image} by combining three techniques: the theory of congruences between Hilbert cuspforms, the
explicit approach from \cite{dieulefait} and a Sturm bound for Hilbert
modular forms that is developed by Burgos and the second author in
the Appendix \ref{appendix:sturm}.  For $\sigma_2$, we will use the
Lefschetz trace formula and automorphisms on the Calabi-Yau threefold
$\tilde X$ (section \ref{s:geom}).  We collect the necessary data for
a proof of Theorem \ref{main-theorem} in section \ref{s:proof}.

\medskip

\noindent{\bf Acknowledgements:} We would like to thank Lassina
D\'embel\'e for many suggestions concerning computing with Hilbert
modular forms. Also we would like to thank Jos\'e Burgos Gil for his
contribution to the appendix with the proof of a Sturm bound. The
computations of the $a_{\id{p}}$ where done using the Pari/GP system
\cite{PARI2}. 
We would like to thank Bill Allombert for implementing a
routine in PARI  that was not included in the original software
for dealing with elements of small norm under a positive definite
quadratic form.
Particular thanks to the referee for his comments and suggestions.
This project was started when two of us enjoyed the hospitality of Harvard University;
it benefitted from research visits to various other institutions.

\section{Computing the residual image of the Galois representation $\sigma_{\ff,2}$}
\label{s:image}

This section deals with the holomorphic Hilbert newform $\ff$ on $F=\QQ(\sqrt{5})$ of
weight $(2,4)$ and conductor $\id{c}_\ff = (30)$ 
constructed in \cite{consani}.
The aim of this section is to prove that the image of the residual
$2$-adic Galois representation $\bar\sigma_{\mathfrak f,2}$ attached to $\ff$ has image a
$2$-group.
This will enable us to apply methods for even trace Galois representations
to $\sigma_{\mathfrak f,2}$.

\subsection{Properties of the Hilbert modular form $\ff$.}

Let us recall some definitions of Hilbert modular forms. For $\id{c}
\subset \Om_F$, let $\Gamma_0(\id{c})$ be the subgroup of
$\SL_2(\Om_F)$ whose second row and first column entry is divisible by
$\id{c}$.  Let $\Siegel$ denote the Poincare upper half plane. Then
given $\vec{k}=(k_1,k_2)$, with $k_1 \equiv k_2 \pmod 2$, a weight
$\vec{k}$ Hilbert modular form of level $\id{c}$ is an holomorphic
function $f:\Siegel^2 \mapsto \CC$ such that for
$\gamma=\left(\begin{smallmatrix} a&b\\ c&d\end{smallmatrix} \right)
  \in \Gamma_0(\id{c})$,
\[
f(\gamma\cdot z_1,\tau(\gamma) \cdot z_2) = f(z_1,z_2) (cz_1+d)^{k_1} (\tau(c)z_2 + \tau(d))^{k_2},
\]
with the usual holomorphicity condition at the cusps. We denote such
space $M_{\vec{k}}(\id{c})$, and $S_{\vec{k}}(\id{c})$ the subspace of
cuspidal forms.

We start studying some properties of the form $\ff$. In sections $6$
and $7$ of \cite{consani} such form was constructed using Eichler's
method on definite quaternion algebras and in Theorem $8.3$ of
loc.~cit.~it is proved that its coefficient field is exactly
$\QQ(\sqrt{5})$.

Since the primes $2$ and $3$ divide the level of the form $\ff$ to the
first power, the automorphic representation attached to $\ff$ is
Steinberg at both primes. To know its behaviour at the prime
$\sqrt{5}$, we consider its twist by a suitable character.

\begin{lemma}
There exists a unique non-trivial quadratic Hecke character
$\chi_{\sqrt{5}}$ of $\Om_F$ (of infinity type $(\sign,\sign)$), whose
conductor is $\sqrt{5}$. The quadratic twist of $\ff$ by
$\chi_{\sqrt{5}}$ corresponds to a Hilbert newform of level
$6\sqrt{5}$ and weight $(2,4)$ which we denote by $\ff\otimes
\chi_{\sqrt{5}}$.
\label{lemma:twist}
\end{lemma}

\begin{proof}
For any prime ideal $\id{p}$ of $\Om_F$, whose residue field
has prime order $p$, we can consider the quotient map 
\[
\Om_F \twoheadrightarrow \Om_F/\id{p} \simeq \ZZ/p\ZZ,
\]
to get a Dirichlet character $\chi_{\id{p}}$ in $\Om_F$. To get the
Hecke character we just need the infinity characters, but note that
the infinite type $(\sign^{\epsilon_1},\sign^{\epsilon_2})$ is
uniquely defined by the conditions
\begin{eqnarray*}
\chi_{\id{p}}(-1) &=& (-1)^{\epsilon_1+\epsilon_2},\\
\chi_{\id{p}}\left(\frac{1+\sqrt{5}}{2}\right) &=& (-1)^{\epsilon_2}.
\end{eqnarray*}
The uniqueness comes from the fact that the fundamental unit is not
totally positive (which is equivalent to say that the class number and
the narrow class number are the same). In our case,
$\chi_{\sqrt{5}}\left(\frac{1+\sqrt{5}}{2}\right) = -1$, and
$\chi_{\sqrt{5}}(-1)=1$ so the first assertion follows.


For the second statement of the Lemma, it is clear (as in the
classical case) that the twist will have trivial character and level at
most $30$ (see for example~\cite[Proposition 4.4]{shimura-hilbert}), so
the proof goes by elimination. The data used was supplied by Lassina
D\'embel\'e, but is now available in the new versions of MAGMA for
example. Here is a resume:
\begin{itemize}
\item The space $S_{(2,4)}(30)$ is of dimension $74$. 
\item Its subspace of newforms is of dimension $22$.
\item There are $16$ eigenforms in the new subspace whose coefficient
  field is $F$ and $2$ eigenforms whose coefficient field has degree
  $3$ over it.
\item The space $S_{(2,4)}(6\sqrt{5})$ is of dimension $14$.
\item Its subspace of newforms is of dimension $4$.
\item There are $4$ eigenforms in the new subspace whose coefficient field is $F$. 
\end{itemize}
By computing the Hecke eigenvalue at the primes above $11$ in the
newspace of level $30$, we find a unique form matching the eigenvalue
of $\ff$ (in version 2.16 of Magma, the form is the first one to
appear), so the form $\ff \otimes \chi_{\sqrt{5}}$ does not lie in the
newspace of level $30$. Since the primes of norm $11$ are generated by
$\frac{7+\sqrt{5}}{2}$ and $\frac{7-\sqrt{5}}{2}$, and
$\chi_{\sqrt{5}}\left(\frac{7+\sqrt{5}}{2}\right)=1$, we search for a
form with the same eigenvalue in level $6\sqrt{5}$ and find a unique
one (the third one), but none in level $6$, hence $\ff \otimes
\chi_{\sqrt{5}}$ has level $6\sqrt{5}$.
\end{proof}

This implies the following.

\begin{coro}
The form $\ff$ is also Steinberg at the prime $\sqrt{5}$.
\end{coro}

Since the form $\ff$ and $\ff \otimes \chi_{\sqrt{5}}$ are congruent modulo
$2$, we can work with the latter form which has smaller level. 

\subsection{Properties of the residual image of $\sigma_{\ff,2}$}

For proving that the image of the residual $2$-adic Galois
representation attached to $\ff$ has image a $2$-group, eventually we
will pursue a similar approach as in \cite{dieulefait}.  We start by
computing all subgroups of $A_5$.

\begin{lemma} Any proper subgroup of $A_5$ is isomorphic to one of the
  following:
\[
\{\{1\}, C_2, C_3, C_2 \times C_2, C_5, S_3, D_5, A_4\}.
\]
\label{lemma:cases}
\end{lemma}

There is an easy classification of the orders of the elements in
$\SL_2(\F_4)$ in terms of the traces.

\begin{lemma}
\label{lemma:group}
If $M \in \SL_2(\FF_4)$, then
\[
\text{ord}(M)=
\begin{cases}
1 & \text{ if } M = \identity,\\
2 & \text{ if } \trace(M)=0 \text{ and } M \neq \identity,\\
3 & \text{ if } \trace(M)=1,\\
5 & \text{ if } \trace(M) \not \in \FF_2.
\end{cases}
\]
\end{lemma}

Recall how to derive a Fourier expansion at $\infty$ for a Hilbert
modular form over $F$.  Let $\tau$ denote the generator of
$\Gal(F/\Q)$.  An element $\nu\in F$ is called totally positive if
both $\nu>0$ and $\tau(\nu)>0$.  We denote this by $\nu\gg 0$.  Since
$F$ has strict class number one, any Hilbert modular form $G$ over $F$
has a $q$-expansion 

\begin{equation}
G(z_1,z_2) = \sum_{\substack{\xi \in \frac{\Om_F}{\sqrt{5}}
   \\ \xi\gg 0}}a_\xi \exp(\xi z_1
+\tau(\xi)z_2),
\label{eq:idealqexp}
\end{equation}
where $\exp(z) = e^{2\pi i z}$. 

Since the coefficient field for $\ff$ is $F=\QQ(\sqrt{5})$ and $2$ is
inert in $F$, we can reduce our coefficients modulo $2$ to get
a $q$-expansion with coefficients in $\FF_4$. We denote by
$\overline{a_\xi}$ the reduced coefficients in $\FF_4$, and we normalize the
$q$-expansion of $\ff$ such that its coefficient
$a_{\frac{1+\sqrt{5}}{2\sqrt{5}}}=1$.

\begin{lemma}
Suppose that some coefficient $\overline{a_\xi}$ is not in $\FF_2$. Then
the $q$-expansion of the form $\ff^3$ has a coefficient which is not
in $\FF_2$ either.
\end{lemma}

\begin{proof} The $q$-expansion of $\ff^3$ is again of the form 
\[
\sum_{\substack{\xi \in \frac{\Om_F}{\sqrt{5}}
   \\ \xi\gg 0}}b_\xi \exp(\xi z_1+\tau(\xi)z_2).
\]
We order the monomials of such $q$-expansion using the total order
given by: $\exp(\xi z_1 + \tau(\xi)z_2) < \exp(\nu z_1+\tau(\nu)z_2)$
iff $\trace(\xi) < \trace(\nu)$ or $\trace(\xi)=\trace(\nu)$ and $\xi
< \nu$. It is easy to check that this gives a total order, and that
this order behaves well under addition.

With this order, the first coefficients of $\ff$ are $a_0$ and
$a_{\frac{1+\sqrt{5}}{2\sqrt{5}}}$, which are $0$ and $1$. Suppose that
  $\overline{a_{\xi_0}}$ is the first coefficient of $\ff$ in $\FF_4$
  and not in $\FF_2$, and lets look at the coefficient with exponent
  $\xi_0+2\left(\frac{1+\sqrt{5}}{2\sqrt{5}}\right)$ of the reduction
  of $\ff^3$. Since all non-zero coefficients have trace at least $1$,
  it is of the form
\[
3 a_{\xi_0} a_{\frac{1+\sqrt{5}}{2\sqrt{5}}}^2 + \text{ product of
     coefficients smaller than } a_{\xi_0},
\]
so the claim follows.
\end{proof}

\begin{remark}
From the proof of the last Lemma, it is clear that if the reduction of
all Fourier coefficients of $\ff$ with trace smaller than a constant
$N$ lie in $\FF_2$, the same happens to $\ff^3$.
\end{remark}

Recall that $\Aut(\CC)$ acts on the space of Hilbert modular forms,
just by acting on Fourier expansions. The following result is due to
Shimura (see Proposition 1.2 of \cite{shimura-hilbert}).

\begin{prop}
Let $\sigma \in \Aut(\CC)$, then $\sigma(M_{\vec{k}}) = M_{\vec{l}}$, where $\vec{l} = \vec{k}^\sigma$.
\label{prop:shimuraconjugation}
\end{prop}


\begin{prop} 
The residual image of $\sigma_{\ff,2}$ is not $D_5$ nor $A_5$.
\label{prop:order5}
\end{prop}

\begin{proof}
Suppose the residual image of $\sigma_{\ff,2}$ is $D_5$ or $A_5$. The
groups $D_5$ and $A_5$ both have order $5$ elements, so by
Lemma~\ref{lemma:group} there exists a prime $\id{p}_0$ such that the
$\id{p}_0$-th Hecke eigenvalue $a_{\id{p}_0}$ of $\ff$ lies in $\F_4$,
but not in $\F_2$. The main idea is that in this case $\ff$ and its
Galois conjugate, $\tau(f)$, have different $q$-expansion modulo
$2$. We would like to consider the difference of these functions, and
prove using a Sturm bound result that this function is zero. The only
problem is that $\ff$ and $\tau(\ff)$ have different weights, so their
difference is not a Hilbert modular form. We overcome this problem by
multiplying the reduction of the cubic powers of these forms by 
appropiate Hasse invariants which give parallel weight $10$ forms
with $q$-expansions which are not in $\FF_2$.

Since the prime $2$ is inert in $\QQ(\sqrt{5})$, there exist forms
$h_1$ and $h_2$ of level $1$ and weight $(2,-1)$ and $(-1,2)$ in
$\overline{\FF}_2$ which are called the partial Hasse invariants (see
\cite{Go}, Definition 3.4, page 179) whose $q$-expansion at any cusp
is $1$. So we consider the
form $(\overline{\ff\otimes \chi_{\sqrt{5}}})^3h_1^2 +
(\overline{\tau(\ff\otimes \chi_{\sqrt{5}}})^3h_2^2$, which is a
parallel weight $10$ form for $\Gamma_0(6\sqrt{5})$ in
$\overline{\FF}_2$, whose $q$-expansion vanishes with order $3$ at all
cusps and is not the zero form (because by assumption at least one coefficient of
the $q$-expansion of $(\ff\otimes \chi_{\sqrt{5}})^3$ is not in
$\FF_2$).


We computed all Hecke eigenvalues with ideals generated by an element
of trace smaller than $181$ of $\ff$ and checked that they all lie in
$2\Om_F$ (a table for such eigenvalues can be found at
\cite{table}). In particular, all the Fourier coefficients of the
Hilbert modular form $\ff$ with trace at most $181$ lie in
$\FF_2$, so the same holds for $(\overline{\ff\otimes
  \chi_{\sqrt{5}}})^3h_2^2$ which implies that the form
$(\overline{\ff\otimes \chi_{\sqrt{5}}})^3h_2^2 +
(\overline{\tau(\ff\otimes \chi_{\sqrt{5}}})^3h_1^2$ has zero Fourier
expansion for all coefficients with trace at most $181$. The
Sturm bound (Theorem~\ref{thm:Sturm}) implies that it must be the zero
form, contradicting our assertion above and thus the original assumption.
\end{proof}

It remains to prove that the residual image at $2$ cannot be any of
the groups $\{C_3, C_5, S_3, A_4\}$. We recall some well known results from
Class Field Theory:

\begin{thm}
\label{thm:ab}
If $L/F$ be an abelian Galois extension unramified outside the set of
places $\set{\id{p}_i}_{i=1}^n$ then there exists a modulus $\id{m} =
\prod_{i=1}^n \id{p}_i^{e(\id{p}_i)}$ such that $\Gal(L/F)$
corresponds to a subgroup of the ray class group $Cl(\Om_F,\id{m})$.
\end{thm}

A bound for $e(\id{p})$ is given by the following result.

\begin{prop}
Let $L/F$ be an abelian Galois extension of prime degree $p$.
Consider a modulus $\id{m} = \prod_{i=1}^n \id{p}_i^{e(\id{p}_i)}$
associated to the extension $L/F$ by Theorem \ref{thm:ab}.  If
$\id{p}$ ramifies in $L/F$, then
\[
\left \{\begin{array}{cl}
e(\id{p}) = 1 & \text{if } \id{p} \nmid p\\
2
\le e(\id{p}) \le \left \lfloor\frac{p e(\id{p}|p)}{p-1} \right \rfloor +1 & \text{if }\id{p} \mid
 p,
\end{array} \right.
\]
where $\id{p}$ is a prime above the rational prime $p$ and
$e(\id{p}|p)$ is the ramification index of $\id{p}$ in $F/\Q$.
\label{cohen}
\end{prop}

\begin{proof} See \cite{cohen} Proposition $3.3.21$ and Proposition $3.3.22$.
\end{proof}

Using these two results, we can compute for each possible Galois group all Galois
extensions of $F$ unramified outside $\{2,3,5\}$.
In each extension, we find a prime
$\id{p}$  where the Frobenius has non-zero trace (in
$\FF_4$). If $a_{\id{p}} \equiv 0 \pmod 2$, for all such primes we are
done:

\begin{prop}
\label{prop:mod}
The residual representation $\bar{\sigma}_{\ff,2}$ has
  image a $2$-group.
\end{prop}

\begin{proof}

\noindent Consider the different cases:

\medskip

$\bullet$ The group $A_4$ has $C_2 \times C_2$ as a normal subgroup
(with generators $(12)(34)$ and $(13)(24)$). The quotient by this
subgroup is a cyclic group of order $3$, so the extension contains a
cubic Galois subfield. Since there are no elements of order $6$ in
$A_4$, a non trivial element in this Galois group will have odd
trace. Thus the case $A_4$ and $C_3$ can be discarded at the same
time.

In order to do so, we can take $\id{m} = 2 \cdot 3^2 \cdot \sqrt{5}$
as the maximal modulus by Proposition \ref{prop:mod}. The ray class
group $Cl(\Om_F,\id{m})$ is isomorphic to $C_{12} \times C_6$ so there
are $4$ cubic extensions ramified at these primes. We consider the
characters as additive characters (by taking logarithms), and denote
by $\psi_1$ and $\psi_2$ two characters that generate the group of
characters of order $3$ (we take the fourth and the second power of
the characters in the previous basis). Instead of computing a prime
ideal where each character is non-zero, we compute two prime ideals
$\id{p}_1$ and $\id{p}_2$ such that
\[
\<(\psi_1(\id{p}_1),\psi_2(\id{p}_1)),(\psi_1(\id{p}_2),\psi_2(\id{p}_2))>
= \ZZ/3\ZZ \times \ZZ/3\ZZ.
\]

Then Proposition $5.4$ of \cite{dieulefait} implies that any cubic
character is non-trivial in one of these two ideals. The two ideals
above the prime $11$ have values $(1,0)$ and $(1,2)$, which is a basis
for $\FF_3 \times \FF_3$. Since the modular form has even traces (in
$\FF_2$) at both primes (see Table of \cite{consani}), we conclude
that the residual representation cannot have image isomorphic to $C_3$
nor $A_4$.

\medskip

$\bullet$ To discard the $C_5$ case, we take $\id{m} = 2\cdot 3\cdot
(\sqrt{5})^3$. The ray class group for this module is cyclic of order
$10$. The generator at primes above $11$ has value $9$, in particular
the order $5$ characters do not vanish at this primes. But the trace
of Frobenius lies in $\FF_2$ at these primes (as mentioned in the
previous case), so the image cannot be cyclic of order $5$.

\medskip

$\bullet$ To discard the $S_3$ case, we start by computing all the
quadratic extensions of $F$ ramified outside the set of primes
$\{2,3,5\}$. The modulus in this case is $\id{m} = 2^3 \cdot 3 \cdot
\sqrt{5}$. The ray class group is isomorphic to $C_4 \times C_4 \times
C_2\times C_2 \times C_2$ so there are $31$ such extensions. In
Table~\ref{table:extensions} we put all the information of these
extensions; the first column has an equation for each such extension,
the second column its discriminant over $\Q$, the third column the
modulus considered (where by $\id{p}_2$ (respectively $\id{p}_5$) we
denote the unique prime ideal in the extension above the rational
prime $2$ (respectively $5$)), the fourth column the ray class group
and the last column the rational primes whose prime divisors in $F$ generate
the $\FF_3$ vector space of cubic characters.

\begin{table}
\scalebox{0.72}{
\begin{tabular}{|r|r|r|r|r|}
\hline
Equation over $F$ & Disc. over $\Q$ & Modulus & Ray Class Group & Rational Primes\\
\hline
$x^2 - 6\sqrt{5}$ & $-2^6 \cdot 3^2\cdot 5^3$ &  $9\cdot \id{p}_2\cdot \id{p}_5$ & $ C_{72}\times
  C_{6}\times C_{3}\times C_{3}$ &  $\{7,11,13\}$\\
$x^2 + 6\sqrt{5}$ & $- 2^6\cdot 3^2\cdot 5^3$ & $9\cdot \id{p}_2\cdot \id{p}_5$ & $ C_{72}\times
  C_{6}\times C_{3}\times C_{3}$ &  $\{7,11,13\}$\\
$x^2+\sqrt{5}+1$ & $-2^6\cdot 5^2$ & $9\cdot \id{p}_2\cdot\sqrt{5}$ & $ C_{12}\times C_{6}\times C_{3}$ &  $\{7,11,13,19\}$\\
$x^2-\sqrt{5}-1$ & $- 2^6\cdot 5^2$ & $9\cdot \id{p}_2\cdot \sqrt{5}$ & $ C_{12}\times C_{6}\times C_{3}$ &  $\{7,11,13,19\}$\\
$x^2-3\sqrt{5}-3$ & $-2^6\cdot 3^2\cdot 5^2$ & $9\cdot \id{p}_2\cdot \sqrt{5}$ & $ C_{36}\times C_{6}\times C_{3}\times C_{3}$ &  $\{7,11,13,17\}$\\
$x^2 + 3\sqrt{5} + 3$ & $- 2^6\cdot 3^2\cdot 5^2$ & $9\cdot \id{p}_2\cdot \sqrt{5}$ & $ C_{36}\times C_{6}\times C_{3}\times C_{3}$ &  $\{7,11,13,17\}$\\
$x^2+2\sqrt{5}$ & $- 2^6\cdot 5^3$ & $9\cdot \id{p}_2\cdot \id{p}_5$ & $ C_{24}\times C_{6}\times C_{3}$ & $\{7,11,13\}$\\
$x^2 - 2\sqrt{5}$ & $- 2^6\cdot 5^3$ & $9\cdot \id{p}_2\cdot \id{p}_5$ & $ C_{24}\times C_{6}\times C_{3}$ &  $\{7,11,13\}$\\
$x^2 + \frac{3}{2}\sqrt{5} + \frac{3}{2}$ & $- 2^4\cdot 3^2\cdot 5^2$ & $9\cdot \id{p}_2\cdot \sqrt{5}$ & $C_{36}\times C_{6}\times C_{3}$ &  $\{7,11,13\}$\\
$x^2 -\frac{3}{2}\sqrt{5} - \frac{3}{2}$ & $- 2^4\cdot 3^2\cdot 5^2$ & $9\cdot \id{p}_2\cdot \sqrt{5}$ & $C_{36}\times C_{6}\times C_{3}$ &  $\{7,11,13\}$\\
$x^2 - \sqrt{5}$ & $- 2^4\cdot 5^3$ & $9\cdot \id{p}_2\cdot \id{p}_5$ & $C_{24}\times C_{6}\times C_{3}$ &  $\{7,11,13\}$\\
$x^2 + \sqrt{5}$ & $- 2^4\cdot 5^3$ & $9\cdot \id{p}_2\cdot \id{p}_5$ & $C_{24}\times C_{6}\times C_{3}$ &  $\{7,11,13\}$\\
$x^2 + 3\sqrt{5}$ & $- 2^4\cdot 3^2\cdot 5^3$ & $9\cdot \id{p}_2\cdot \id{p}_5$ &$C_{24}\times C_{6}\times C_{3}\times C_{3}$ &  $\{7,11,13\}$\\
$x^2 - 3\sqrt{5}$ & $- 2^4\cdot 3^2\cdot 5^3$ & $9\cdot \id{p}_2\cdot \id{p}_5$ &$C_{24}\times C_{6}\times C_{3}\times C_{3}$ &  $\{7,11,13\}$\\
$x^2 -\frac{1}{2}\sqrt{5} - \frac{1}{2}$ & $- 2^4\cdot 5^2$ & $9\cdot \id{p}_2\cdot \sqrt{5}$ & $C_{12}\times C_{6}\times C_{3}\times C_{3}$ &  $\{7,11,13,17\}$\\
$x^2 + \frac{1}{2}\sqrt{5} + \frac{1}{2}$ & $- 2^4\cdot 5^2$ & $9\cdot \id{p}_2\cdot \sqrt{5}$ &$C_{12}\times C_{6}\times C_{3}\times C_{3}$ &  $\{7,11,13,17\}$\\
$x^2+\sqrt{5}+5$ & $2^6\cdot 5^3$ & $9\cdot \id{p}_2\cdot \id{p}_5$ &$C_{60}\times C_{6}\times C_{3}\times C_{3}$ &  $\{7,11,13,23\}$\\
$x^2-6$ & $2^6\cdot 3^2\cdot 5^2$ & $9\cdot \id{p}_2\cdot \sqrt{5}$ & $C_{12}\times C_{6}\times C_{6}$ &  $\{7,11,13,23\}$\\
$x^2+2$ & $2^6\cdot 5^2$ & $9\cdot \id{p}_2\cdot \sqrt{5}$ &$ C_{24}\times C_{12}\times C_{3}\times C_{3}\times C_{3}$ &  $\{7,11,13,23\}$\\
$x^2-3\sqrt{5}-15$ & $2^6\cdot 3^2\cdot 5^3$ & $9\cdot \id{p}_2\cdot \id{p}_5$ & $C_{12}\times C_{6}$ &  $\{7,11,13\}$\\
$x^2+6$ & $2^6\cdot 3^2\cdot 5^2$ & $9\cdot \id{p}_2\cdot \sqrt{5}$ &$C_{72}\times C_{12}\times C_{3}\times C_{3}\times C_{3}$ &  $\{7,11,13,17\}$\\
$x^2 -\sqrt{5} - 5$ & $2^6\cdot 5^3$ & $9\cdot \id{p}_2\cdot \id{p}_5$ & $C_{36}\times C_{18}\times C_{3}\times C_{3}$ &  $\{7,11,13\}$\\
$x^2 + 3\sqrt{5} + 15$ & $2^6\cdot 3^2\cdot 5^3$ & $9\cdot \id{p}_2\cdot \id{p}_5$ & $C_{12}\times C_{6}\times C_{6}$ &  $\{7,11,13,61\}$\\
$x^2 - 2$ & $2^6\cdot 5^2$ & $9\cdot \id{p}_2\cdot \sqrt{5}$ & $C_{36}\times C_{18}\times C_{3}\times C_{3}$ &  $\{7,11,13\}$\\
$x^2 + \frac{3}{2}\sqrt{5} + \frac{15}{2}$ & $2^4\cdot 3^2\cdot 5^3$ & $9\cdot \id{p}_2\cdot \id{p}_5$ & $C_{36}\times C_{18} \times C_3 \times C_3$ &  $\{7,11,13\}$\\
$x^2 + 3$ & $3^2\cdot 5^2$ & $18\cdot \sqrt{5}$ & $C_{36}\times C_{3}\times C_{3}\times C_{3}\times C_{3}\times C_{3}$ &  $\{7,11,13,17,19,23\}$\\
$x^2 + 1$ & $2^4\cdot 5^2$ & $9\cdot \id{p}_2\cdot \sqrt{5}$ & $C_{24}\times C_{12}\times C_{3}\times C_{3}$ &  $\{7,11,13,17\}$\\
$x^2 + \frac{1}{2}\sqrt{5} + \frac{5}{2}$ & $5^3$ & $18\cdot \id{p}_5$ & $C_{60}\times C_{3}\times C_{3}\times C_{3}$ &  $\{7,11,13\}$\\
$x^2 -\frac{1}{2}\sqrt{5} - \frac{5}{2}$ & $2^4\cdot 5^3$ & $9\cdot \id{p}_2\cdot \id{p}_5$ & $C_{12}\times C_{6}$ &  $\{7,11,13\}$\\
$x^2 - 3$ & $2^4\cdot 3^2\cdot 5^2$ & $9\cdot \id{p}_2\cdot \sqrt{5}$ & $C_{12}\times C_{6}\times C_{6}\times C_{3}$ &  $\{7,11,13\}$\\
$x^2 -\frac{3}{2}\sqrt{5} - \frac{15}{2}$ & $3^2\cdot 5^3$ & $18\cdot \id{p}_5$ &$C_{6}\times C_{6}$ &  $\{7,11,13\}$\\
\hline
\end{tabular}
}
\caption{Quadratic extensions of $F$ unramified outside $\{2,3,5\}$ \label{table:extensions}}
\end{table}

Note that the first $16$ extensions are not Galois over $\Q$ and are
listed so that they are isomorphic in pairs. The last $15$ are indeed
Galois over $\Q$. This implies that we could consider one element of
each pair for the first fields. Note that the primes at conjugate
fields are the same, since the Frobenius at a prime in $F$ of a field
is the same as the Frobenius at the conjugate of the prime in the
conjugate extension.

The traces of Frobenius are even in all such primes (see Table of
\cite{consani}), so we conclude that the image of the residual
representation $\bar{\sigma}_{\ff,2}$ cannot be isomorphic to $S_3$
either.
\end{proof}

\section{Computing the residual image of the Galois representation $\sigma_2$}
\label{s:geom}

In this section we consider the Consani-Scholten
quintic $\tilde X$ from \cite{consani}.
Our goal is to prove that the $\ell$-adic Galois
representations $\rho$ of $\HH^3(\tilde X_{\bar\Q},\Q_\ell)$ have $4$-divisible trace.
This will
be used in \ref{ss:sigma} to deduce that modulo $2$ the restricted $2$-adic two-dimensional Galois
representations have image in $\SL_2(\F_2)$,
and in fact even trace, so that we can apply
an adaptation of the  Faltings-Serre-Livn\'e method
in order to prove Theorem \ref{main-theorem}.

\subsection{Setup}

Consider the Chebyshev polynomial
\[
P(y, z) = (y^5+z^5)-5yz(y^2+z^2)+5yz(y+z)+5(y^2+z^2)-5(y+z).
\]
Then we define an affine variety $X$ in $\mathbb{A}^4$ by
\[
X:\;\;\; P(x_1, x_2) = P(x_4,x_5).
\]
Let $\bar X\subset\PP^4$ denote the projective closure of $X$. Then $\bar X$ has 120 ordinary double points. Let $\tilde X$ denote a desingularisation obtained by blowing up $\bar X$ at the singularities.

\begin{remark}
We might also consider a small resolution $\hat X$, as many of the nodes lie on products of lines. Then we would have to check that $\hat X$ is projective and can be defined over $\Q$. This desingularisation would have the advantage of producing an honest Calabi-Yau threefold, but it does not affect the question of Hilbert modularity.
\end{remark}

Consani and Scholten compute the Hodge diamond of $\tilde X$ as follows:

\[
\begin{matrix}
&&& 1 &&&\\
&& 0 && 0 &&\\
& 0 && 141 && 0 &\\
1 && 1 && 1 && 1\\
& 0 && 141 && 0 &\\
&& 0 && 0 &&\\
&&& 1 &&&
\end{matrix}
\]

Hence the \'etale cohomology groups $\HH^3(\tilde X_{\bar\Q}, {\Q}_\ell)$ ($\ell$ prime) give rise to a compatible system of four-dimensional Galois representations $\{\rho_\ell\}$. Since $\tilde X$ has good reduction outside $\{2, 3, 5\}$, $\rho_\ell$ is unramified outside $\{2, 3, 5, \ell\}$.

Let $F=\Q(\sqrt{5})$ and fix some prime $\ell\in\NN$ and a prime $\lambda$ of $F$ above $\ell$. Then Consani and Scholten prove for the $\ell$- resp.~$\lambda$-adic representations:

\begin{thm}[Consani-Scholten {\cite{consani}}]
\label{Thm:CS}
The restriction $\rho|_{\Gal(\bar\Q/F)}$ is reducible as a representation into $\GL_4(F_\lambda)$: There is a Galois representation
\[
\sigma: \Gal(\bar\Q/F) \to \GL_2(F_\lambda)
\]
such that $\rho=\Ind_F^\Q \sigma$.
\end{thm}

The theorem implies in particular that internal and external conjugation have the same effect on $\sigma$.
Here we want to prove the following property:
\begin{prop}\label{Prop:4}
The Galois representation $\rho$ has $4$-divisible trace,
and so has any restriction to a finite extension of $\Q$.
That is, 
$\rho(\Frob_q)\equiv 0\mod 4$
for any odd prime power $q$.
\end{prop}

In fact, for $q\equiv 2,3\mod 5$, Consani-Scholten proved that
$\rho(\Frob_q)$ has zero trace. Hence we would only have to consider
the case $q\equiv 1,4 \mod 5$, although we will treat the problem in
full generality.

As a corollary, we will deduce that $\sigma$ is even in \ref{ss:sigma}
as required for the proof of Theorem \ref{main-theorem}.

\subsection{Lefschetz fixed point formula}

Choose a prime $p\neq\ell$ of good reduction for $\tilde X$ and let
$q=p^r$. Consider the geometric Frobenius endomorphism Frob$_q$ on
$\tilde X/\F_p$, raising coordinates to their $q$-th powers. Then the
Lefschetz fixed point formula tells us that
\begin{eqnarray*}
\# \tilde X(\F_q) & = & \sum_{i=0}^6 (-1)^i \, \text{trace Frob}_q^*(\HH^i(\tilde X_{\bar\Q}, \Q_\ell)).
\end{eqnarray*}
In our situation, this simplifies as follows: $h^1=h^5=0$; $\HH^2(\tilde X)$ and $\HH^4(\tilde X)$ are algebraic by virtue of the exponential sequence and Poincar\'e duality. Moreover Frob$_q^*$ factors through a permutation on $\HH^2(\tilde X)$, i.e.~all eigenvalues have the shape $\zeta q$ where $\zeta$ is some root of unity. Denote the sum of these roots of unity by $h_q$
(which is an integer in $\Z$ by the Weil conjectures). Finally geometric and algebraic Frobenius are compatible through $\rho$. Hence
\begin{eqnarray}\label{eq:trace}
t_q = \text{trace}\; \rho(\Frob_q) = 1 + h_q \,q (1+q) + q^3 - \# \tilde X(\F_q).
\end{eqnarray}

Prop.~\ref{Prop:4} claims that the left hand side is divisible by $4$. If $q\equiv-1\mod 4$, this is a consequence of the following
\begin{lemma}\label{Lem:number}
For any good prime $p$ and  $q=p^r$, $\# \tilde X(\F_q)\equiv 0\mod 4$.
\end{lemma}

If $q\equiv 1\mod 4$, then we furthermore need the following
\begin{lemma}\label{Lem:trace}
For any good prime $p$ and  $q=p^r$, $h_q$ is odd.
\end{lemma}


\subsection{Proof of Lemma \ref{Lem:number}}

To prove Lemma \ref{Lem:number}, we use the action of the dihedral group $D_4$ on $\tilde X$ and the knowledge about the exceptional divisors from Consani-Scholten.

Let  $\zeta_n$ denote a primitive $n$-th root of unity. Then all the nodes are defined over $\Q(\zeta_{15})$. A detailed list can be found in \cite{consani}. Over the field of definition of the node, the exceptional divisor $E$ is isomorphic to $\PP^1\times\PP^1$. Hence $\# E(\F_q) = (q+1)^2$ if the node is defined over $\F_q$.

\begin{lemma}\label{Lem:exceptional}
For any good prime $p$ and $q=p^r$, $\# \tilde X(\F_q) \equiv \# \bar X(\F_q) \mod 32$.
\end{lemma}

\emph{Proof:} By \cite{consani}, we have
\[
\# \{\text{nodes over } \F_q\} = \begin{cases}
0, & q\equiv 2,7,8,13\mod 15,\\
8, & q\equiv 14 \mod 15,\\
24, & q\equiv 4 \mod 15,\\
104, & q\equiv 11\mod 15,\\
120, & q\equiv 1\mod 15.
\end{cases}
\]
Since the number of points on the exceptional divisor is the same for all nodes defined over $\F_q$, the claim follows. \qed

\begin{lemma}\label{Lem:oo}
For any good prime $p$ and $q=p^r$, $\# \bar X(\F_q) \equiv \# X(\F_q) - q \mod 4$.
\end{lemma}

\emph{Proof:} The affine variety $X$ is compactified by adding a smooth surface at $\infty$. In fact, this is the Fermat surface of degree five:
\[
S=\{x_0^5+x_1^5-x_3^5-x_4^5=0\}\subset\PP^3.
\]
Hence $\# \bar X(\F_q) -  \# X(\F_q) = \# S(\F_q).$ Thus Lemma \ref{Lem:oo} amounts to the following

\begin{lemma}\label{Lem:Fermat}
For $p\neq 5$ and $q=p^r$, $\# S(\F_q)\equiv 1+q+q^2\mod 4$.
\end{lemma}

The proof of this lemma will be postponed to the end of this subsection. Lemma \ref{Lem:oo} follows. \qed

\medskip

To prove the corresponding statement about the affine variety $X$, we use the action of the dihedral group $D_4$ generated by the involutions
\[
(x_1, x_2, x_3, x_4) \mapsto (x_2, x_1, x_3, x_4),\;\;\; (x_1, x_2, x_3, x_4) \mapsto (x_1, x_2, x_4, x_3)
\]
and by the cyclic permutation
\begin{eqnarray}\label{eq:cycl}
\gamma:\;\;\;(x_1, x_2, x_3, x_4) \mapsto (x_3, x_4, x_2, x_1).
\end{eqnarray}
It follows that
\[
\# X(\F_q)  \equiv  \# \{x\in X(\F_q); \#\{y\in (D_4-\text{orbit of }x)\}<4\} \mod 4.
\]
Here $\{x\in X(\F_q); \#\{y\in (D_4-\text{orbit of }x)\}<4\} = \{(x_1, x_1, x_3, x_3)\in X(\F_q)\}$.
We are led to consider the affine curve $C$ in $\mathbb{A}^2$ defined by
\[
C:\;\;\; P(y,y) = P(z,z).
\]
Then the above subset of $X(\F_q)$ is in bijection with $C(\F_q)$, and we obtain
\begin{eqnarray}\label{eq:C}
\# X(\F_q)\equiv C(\F_q)\mod 4.
\end{eqnarray}

\begin{lemma}\label{Lem:C}
For any good prime $p$ and $q=p^r$, $\# C(\F_q) \equiv q \mod 4$.
\end{lemma}


\emph{Proof of Lemma \ref{Lem:C}:} $C$ is reducible. The change of variables
\[
u=\frac{y+z}2,\;\;\; v=\frac{y-z}2
\]
allows us to write
$$
P(y,y)-P(z,z) = v (v^4 + 5 (2u^2-4u+1)v^2 + 5 (u^2-3u+1)(u^2-u-1))=v G(u,v).
$$
Hence
\begin{eqnarray}\label{eq:C-B}
\# C(\F_q) = q + \# (B(\F_q)\cap\{v\neq 0\}) 
\end{eqnarray}
where $B$ is the affine curve in $\mathbb{A}^2$ given by $G(u, v)=0$. Here $B$ is endowed with involutions
\[
(u, v) \mapsto (u, -v), \;\;\;\ (u, v)\mapsto (2-u, v).
\]
For the number of points, this implies
\begin{eqnarray*}
\# (B(\F_q)\cap\{v\neq 0\}) & \equiv & \# (B(\F_q) \cap \{u=1, v\neq 0\}) \mod 4\\
& = & \# \{v\in\F_q; v^4-5v^2+5=0\}.
\end{eqnarray*}
The last polynomial factors as
\[
4 (v^4-5v^2+5) = (2v^2-5-\sqrt 5)\,(2v^2-5+\sqrt 5).
\]
Since $\frac{5+\sqrt 5}2\cdot\frac{5-\sqrt 5}2=4$, a square,
we deduce that the last equation has either zero or four solutions in $\F_q$. In particular, (\ref{eq:C-B}) reduces to $\# C(\F_q) \equiv q\mod 4$, i.e.~to the claim of Lemma \ref{Lem:C}. \qed

\medskip

\emph{Proof of Lemma \ref{Lem:Fermat}:} We shall again use the cyclic permutation $\gamma$ from (\ref{eq:cycl}), but this time it operates on the homogeneous coordinates of $\PP^3$. Hence
\begin{eqnarray}\label{eq:S-fix}
\# S(\F_q) \equiv \# (S\cap \text{Fix}(\sigma^2))(\F_q) \mod 4.
\end{eqnarray}
Here
\[
\text{Fix}(\sigma^2) = \{[\lambda, \mu, \pm\lambda, \pm\mu]; [\lambda,\mu]\in\PP^1\}.
\]
One of these lines is contained in $S$, and it is easy to see that there are exactly $(5, q-1)$ further points of intersection unless $p=2$. I.e.~
\[
\# (S\cap \text{Fix}(\sigma^2))(\F_q) = 1+q+\begin{cases}
0, & p=2\\
(5,q-1), & p\neq 2
\end{cases}
\;\,\equiv 1+q+q^2 \mod 4.
\]
Lemma \ref{Lem:Fermat} follows from this congruence and (\ref{eq:S-fix}). \qed

\medskip

\emph{Proof of Lemma \ref{Lem:number}:}
Lemma \ref{Lem:C} and \eqref{eq:C} imply that $\#X(\F_q)\equiv q\mod 4$.
By Lemma \ref{Lem:oo} this gives $\#\bar X(\F_q)\equiv 0 \mod 4$.
The according statement for $\tilde X$ is obtained from Lemma \ref{Lem:exceptional}.
\qed

\subsection{Proof of Lemma \ref{Lem:trace}}

Lemma \ref{Lem:trace} states that the trace $h_q$ of Frob$_q$ on $H^2(\tilde X_{\bar\Q}, \Q_\ell(1))$ is always odd. We shall first prove the following auxiliary result:

\begin{lemma}\label{Lem:Gal}
The Galois group $\Gal(\bar\Q/\Q(\zeta_{15}))$ acts trivially on $\HH^2(\tilde X_{\bar\Q}, \Q_\ell(1))$.
\end{lemma}

\emph{Proof:} Denote the exceptional locus of the blow-up by $E$.
Then $E$ is defined over $\Q$.
The  Leray spectral sequence for the desingularisation gives an exact sequence
\begin{eqnarray}\label{eq:es}
0 \to \HH^2(\bar X_{\bar\Q}, \Q_\ell(1)) \to \HH^2(\tilde X_{\bar\Q}, \Q_\ell(1)) \to \HH^2(E_{\bar\Q}, \Q_\ell(1)).
\end{eqnarray}
By construction, (\ref{eq:es}) is compatible with the Galois action.
Here $\HH^2(\bar X_{\bar\Q}, \Q_\ell(1))$ is the same as for a general quintic hypersurface in $\PP^4$.
Hence it has dimension one and is generated by the class of a hyperplane section.
In particular, $\Gal(\bar\Q/\Q)$ acts trivially on $\HH^2(\bar X_{\bar\Q}, \Q_\ell(1))$.
Recall that every component of $E$ as well as both rulings on every component are defined over $\Q(\zeta_{15})$.
Hence $\Gal(\bar\Q/\Q(\zeta_{15}))$ acts trivially on $\HH^2(E_{\bar\Q}, \Q_\ell(1))$.
By the Galois-equivariant exact sequence \eqref{eq:es},
the same holds for $\HH^2(\tilde X_{\bar\Q}, \Q_\ell(1))$.
\qed

\medskip

It follows from Lemma \ref{Lem:Gal}, that $h_q=141$ if $q\equiv 1\mod
15$. The prove the parity for the other residue classes, we need two
easy statements about sums of primitive roots of unity. They involve
the M\"obius function $\mu: \NN \to \{-1,0,1\}$:
\[
\mu(n)=\begin{cases}
0, & $n$ \text{ not squarefree},\\
(-1)^m, & $n$ \text{ squarefree with $m$ prime divisors}.
\end{cases}
\]

\begin{lemma}
Let $n\in\NN$ and $\zeta_n$ a primitive $n$-th root of unity. Then $\zeta_n$ has trace $\mu(n)$.
\end{lemma}

The lemma follows immediately from the factorisation of the cyclotomic polynomial $x^n-1$ and the definition of $\mu(n)$.

\begin{lemma}\label{Lem:power}
Let $n\in\NN$ and $\zeta_n$ a primitive $n$-th root of unity. Let $m=2^s\cdot k$ with $(k,n)=1$. Then
\[
\mu(n) = \text{trace } \zeta_n =
\sum_{j\in(\ZZ/n\ZZ)^*} \zeta_n^j \equiv
\sum_{j\in(\ZZ/n\ZZ)^*} \zeta_n^{mj} \mod 2.
\]
\end{lemma}

\emph{Proof:}
If $(m,n)=1$, then taking $m$-th powers permutes the primitive $n$-th roots of unity and both sums coincide. Hence it suffices to consider the case where $m=2^s (s>0)$ and $2|n$.

If $4\nmid n$, then $\{\zeta_n^{mj}; j\in(\ZZ/n\ZZ)^*\}$ is the set of $\frac n2$-th primitive roots of unity. Hence
\[
\sum_{j\in(\ZZ/n\ZZ)^*} \zeta_n^{mj} = \mu\left(\frac n2\right) = -\mu(n)
\]
and the claim follows mod $2$. If $4|n$, then $\mu(n)=0$ and every element in $\{\zeta_n^{mj}; j\in(\ZZ/n\ZZ)^*\}$ appears with multiplicity $(m,n)$. Hence
\[
2 \mid (m,n) \mid \sum_{j\in(\ZZ/n\ZZ)^*} \zeta_n^{mj},
\]
and we obtain the claimed congruence. \qed

\medskip

\emph{Proof of Lemma \ref{Lem:trace}:}
Let $\Xi$ be the set of eigenvalues of Frob$_q$ on $H^2(\tilde X_{\bar\Q}, \Q_\ell(1))$ with multiplicities. Then
\[
h_q=\sum_{\zeta\in\Xi} \zeta.
\]
Recall that the Galois group Gal$(\bar\Q/\Q(\zeta_{15}))$ acts trivially on $H^2(\tilde X_{\bar\Q}, \Q_\ell(1))$. Since $\Q(\zeta_{15})/\Q$ is Galois of degree eight, we deduce that $\zeta^8=1$ for each $\zeta\in\Xi$.
In particular
\begin{eqnarray}
\label{eq:sum}
\sum_{\zeta\in\Xi} \zeta^8=141.
\end{eqnarray}
In the present situation, $h_q\in\ZZ$,
i.e.~$h_q$ is a sum of traces of elements in $\Xi$.
Hence
we can apply Lemma \ref{Lem:power} to deduce that $h_q$ has the same parity as the sum in \eqref{eq:sum}.
That is, $h_q$ is odd. \qed

\subsection{Evenness of $\sigma$}
\label{ss:sigma}

In this subsection we conclude our preparations for the proof of Theorem \ref{main-theorem}
by proving the following corollary of Proposition \ref{Prop:4}:

\begin{coro}
\label{Cor:2}
The Galois representation $\sigma$ is even.
\end{coro}

Recall that $\sigma$ induces the 4-dimensional Galois representation $\rho$ over $\Q$.
The proof commences by spelling out the characteristic polynomial of $\rho(\Frob_q)$
for some odd prime power $q$:
\[
\mu_q(T) = T^4 - t_q T^3 + u T^2 - q^3 t_q T + q^6
\]
with $u = (t_q^2 - t_{q^2})/2$ using the notation of \eqref{eq:trace}.
By Proposition \ref{Prop:4}, $t_q \equiv t_{q^2} \equiv 0 \mod 4$, so 
$u \equiv 0 \mod 2$.

Now we turn to $\sigma$, the 2-dimensional Galois representation with values in 
$F =
\Q(\sqrt 5)$ inducing $\rho$.
Let $\q$ denote  some power of a prime ideal
in $F$ with odd norm $q\in\Z$.
We write
\[
s_\q = \trace\sigma(\Frob_\q)
\]
and consider the case $s_\q \not \in \Z$. 
Then $s_\q=a+b \omega$
where $\omega$ solves $v^2-v-1=0$ and  $a,b \in \Z, b \neq 0$. 
By Theorem \ref{Thm:CS} we have
\[
t_q = 2a+b.
\]
Since $t_q$ is even by Proposition \ref{Prop:4}, 
so is $b$, and we can write somewhat more intuitively
$s_\q = c + d\sqrt 5$ with $c, d \in \Z, d\neq 0$. 
This already implies that the mod $2$-reduction
of $\sigma$ has traces in $\F_2$, so it will have image in $\SL(2,\F_2)$.

In the new notation, we obtain 
\[
t_q = 2c,
\] 
so by Proposition \ref{Prop:4} the input $c$ is even.
But then, factoring $\mu_q(T)$ into quadratic factors over $F$
corresponding to $\sigma$ and its external conjugate, 
the coefficient of $T^2$
reads
\[
u = (c^2-5d^2) + 2q^3.
\]
Since we have already seen that $u$ and $c$ are even, 
we find that $d$ is even, too. That is, $\sigma$
has even trace at $\Frob_\q$.
The case $s_\q \in \Z$ is essentially the same argument, but even simpler.
\qed

\section{Proof of the Main Theorem}
\label{s:proof}

There is version of the Faltings-Serre method in \cite{Livne} that
allows to compare two-dimensional $2$-adic Galois representations
with even traces.  Here we have to modify this approach slightly since
the two-dimensional Galois representation $\sigma_2$ is only
determined up to conjugation of its coefficients in the quadratic field $F$.  While the
original result involved the notion of non-cubic test sets, in order
to prove Theorem \ref{main-theorem}, we replace this notion by
non-quartic sets:

\begin{defi}
A subset $T$ of a finite dimensional vector space $V$ is
\emph{non-quartic} if every homogeneous polynomial of degree $4$ on $V$
which vanishes on $T$, vanishes in the whole $V$.
\end{defi}

The following lemma is useful to lower the cardinality of the test set
$T$.

\begin{lemma}
Let $V$ be a finite-dimensional vector space.  
Let $T$ be a subset of
$V$ which contains $4$ distinct hyperplanes through the origin and a
point outside them.
Then $T\setminus\{0\}$ is non-quartic.
\label{lemma:non-cubic}
\end{lemma}

\begin{proof}
Let $L_1, \hdots, L_4$ denote  linear homogeneous
polynomials giving equations of the four hyperplanes.
Let
$P(x_1,\ldots,x_n)$ be a homogeneous quartic polynomial vanishing on all points of
$T$. Division with remainder gives a representation
\[
P(x_1,\ldots,x_n) = L_1 Q(x_1,x_2,\ldots,x_n) + P_2(x_1,\ldots,x_n)
\]
with $L_1\nmid P_2$.
The crucial property here is the following:
Since $T$ contains the hyperplane $\{L_1=0\}$ and $P$ vanishes on this
hyperplane, $P_2$ vanishes on all of $V$.
(To see this, apply a linear transformation so that $L_1=x_1$;
then $P_2=P_2(x_2,\hdots,x_n)$,
and $P_2$ vanishes on the hyperplane $\{x_1=0\}$ if and only if it vanishes on $V$).
Since the hyperplanes are linearly independent, 
we can apply the same argument to
the other three hyperplanes (starting with $L_1Q$ instead of $P$). We obtain
\[
P(x_1,\ldots,x_n) = A\cdot L_1L_2L_3L_4 + \tilde{P}(x_1,\ldots,x_n),
\]
where $A$ is a field constant and $\tilde{P}$ vanishes identically on $V$. Since
$T$ contains a point outside the union of the four hyperplanes, $A$
must be zero.

But then $P=\tilde P$ vanishes on all of $V$.  Since this argument
applies to any homogeneous quartic polynomial $P$, the test set
$T\setminus\{0\}$ is non-quartic.
\end{proof}

\begin{remark}
Note that Lemma \ref{lemma:non-cubic} does explicitly not require the hyperplanes to be linearly independent.
It is immediate from the proof of Lemma \ref{lemma:non-cubic}
that the same argument works for test sets for homogeneous polynomials of degree $n$ 
if we find $n$ distinct hyperplanes through the origin and a point outside them.
\end{remark}

We want to compare the two Galois representations, $\sigma_2$ and
$\sigma_{\ff,2}$.  It is crucial that in the present situation we know
that the external Galois conjugate representation exists: this follows in the
geometric example by construction, and in the modular example we can
consider the $2$-adic representation attached to the conjugate Hilbert
modular forms $\tau(\ff)$, where $\tau$ is a generator of the group
$\Gal(F/\Q)$. For any given $\ell$-adic Galois representation $\rho$
with field of coefficients $F$ (i.e, the field generated by the traces
of Frobenius elements is $F$) we will denote by $\rho'$ the external conjugate
representation (if we know that such a representation exists). Since
for the Calabi-Yau threefold $\tilde X$, we can only compute the
traces of the $4$-dimensional Galois representation $\sigma_2 \oplus
\sigma'_2$ of $\Gal(\bar{\Q}/F)$ (or actually of $\Gal(\bar\Q/\Q)$),
we will need the following generalization of Theorem 4.3 in
\cite{Livne} about Galois representations whose residual images are
$2$-groups:

%

\begin{thm}
\label{Livne}
 Let $K$ be a global field, $S$ a finite set of primes of $K$ and $E$
 the  unramified quadratic extension of $\Q_2$. Denote by $K_S$ the compositum of all
 quadratic extensions of $K$ unramified outside $S$ and by $\bfP_2$
 the maximal prime ideal of $\Om := \Om_E$. Suppose
 $\rho_1,\rho_2:\Gal(\overline \QQ/K)\to\GL_2(E)$ are continuous
 representations, unramified outside $S$, and with field of coefficients $F$, and assume also that their external Galois
  conjugates exist. We suppose that the following conditions are satisfied:
\begin{enumerate}[1.]
\item $\trace(\rho_1)\equiv\trace(\rho_2)\equiv 0 \pmod {\bfP_2}$ and
  $\det(\rho_1)\equiv\det(\rho_2) \equiv 1 \pmod {\bfP_2}$.
\item There exists a set $T$ of primes of $K$, disjoint from $S$, for which
\begin{enumerate}[(i)]
\item The image of the set $\{\Frob_t\}_{t \in T}$ in $\Gal(K_S/K)\backslash
  \{0\}$ is non-quartic.
\item $\trace(\rho_1(\Frob_t)) + \trace(\rho'_1(\Frob_t))
  =\trace(\rho_2(\Frob_t)) + \trace(\rho'_2(\Frob_t)) $ and
 \break $\det(\rho_1(\Frob_t))=\det(\rho_2(\Frob_t))$ for all $t\in T$.
\end{enumerate}
\end{enumerate}
Then $\rho_1 \oplus \rho'_1$ and $\rho_2 \oplus \rho'_2$ have isomorphic semi-simplifications.
\end{thm}

\begin{proof} This is just a slight generalization of Proposition 4.7 and Theorem
4.3  in \cite{Livne}, we reproduce most of the arguments for the
reader convenience, adapted to our situation. \\
Observe that due to assumption (1), and the fact that similar
conditions hold also for $\rho'_1$ and $\rho'_2$, the image of any
of the representations $\rho_1$, $\rho'_1$, $\rho_2$ and $\rho'_2$
is a pro-$2$-group. Let $G$ be the image of the product of the four
representations. Then, $G$ is a topologically finitely generated
pro-$2$-group and the four representations can be thought of (and we
will do so for the rest of this proof) as representations of $G$,
each being obtained from a suitable projection.\\
Set $M_2$ to be the algebra of $2$ by $2$ matrices with coefficients
in $\Om$. 
For $g \in G$, let $\rho : G \rightarrow M_2 \times M_2$ be the map:
$$ \rho(g) = (\rho_1(g), \rho_2(g) ) $$
Keeping the notation in \cite{Livne}, we call $\Sigma$ the subset of
$G$ corresponding to $T$ (the projection of the elements in
$\{\Frob_t\}_{t \in T}$ to
$G$).\\
Let $M$ be the $\Z_2$-linear span of $\rho(G)$. Since $\Om$ has rank
$2$ over $\Z_2$, $M$ is a sub-algebra with unity of $M_2 \times M_2$
 which is free of rank at most $16$ as a module over $\Z_2$.\\
 We consider $R = M / 2 M$, which is an $\F_2$-vector space of
 dimension at most $16$. Denote the image of $g \in G$ in $R$ by
 $\bar{g}$. Set $\Gamma = \{ \bar{g} \; | \; g \in G \} $. Then $\Gamma
 \subseteq R^{\times}$. 
 $R$ is spanned as $\F_2$-vector space by $\Gamma$ and $\dim_{\F_2}
 R \leq 16$.\\
 We will show the following:

\smallskip

\textbf{Assertion:} $R$ is spanned over $\F_2$ by $\{ \bar{\sigma}
\; |  \; \sigma \in \Sigma \cup \{1 \} \}$.

\smallskip

In order to do so, following \cite{Livne},  we have first to prove
that
\begin{eqnarray}
 \label{eq:sigma}
\sigma \in \Sigma \Rightarrow \bar{\sigma}^2 = \bar{1} \text{ in }
\Gamma.
\end{eqnarray}

Let $\sigma \in \Sigma$. Set $d= \det \rho_1(\sigma) = \det
\rho_2(\sigma)$  (the last equality is due to assumption (2)(ii))
and $ t_1 = \trace \rho_1(\sigma)$, $ t_2
 = \trace \rho_2(\sigma)$.\\
 By the Cayley-Hamilton theorem (for $2$ by $2$ matrices) we have:
$$\rho(\sigma)^2 = ( t_1 \rho_1(g), t_2 \rho_2(g) ) -d(I,I)$$
in $M_2 \times M_2$, where $I$ is the $2$ by $2$ identity matrix.\\
Reducing modulo $\bfP_2$ this equality, since assumption (1) gives
$t_1 \equiv t_2 \equiv 0 \pmod{\bfP_2}$ and $d \equiv 1
\pmod{\bfP_2}$ we obtain $\bar{\sigma}^2 = \bar{1}$ in $\Gamma$, and
this is formula (\ref{eq:sigma}).\\
If we call $G^*$ the Frattini subgroup of $G$, and $\Gamma^*$ the
one of $\Gamma$, since $G/G^* \rightarrow \Gamma / \Gamma^*$ is
surjective, the image of $\Sigma$ in $\Gamma / \Gamma^*$ is
non-quadratic (we use assumption (2)(i) and the fact that
non-quartic easily implies non-quadratic). Then, by Lemma 4.5 in
\cite{Livne} we see that formula (\ref{eq:sigma}) implies that
$\Gamma^* = 1$ and hence $\Gamma \cong (\Z / 2 \Z)^r$, for some
$r$.\\
In particular, $\Gamma$ is commutative, and since $\Gamma$ spans
$R$, $R$ is a commutative ring.\\
Let $r \in R$ be any element and write $r = \sum_{\gamma \in \Gamma}
k_\gamma \gamma$, with $k_\gamma \in \F_2$.\\
Then $r^2 = (\sum k_\gamma^2) \cdot \bar{1}$. Hence $r^2= 0$ or else
$r$ is invertible.\\
It follows that $R$ is a local artinian algebra over $\F_2$ with
maximal ideal
$$ P = \{  r\in R \;  |  \;      r^2 = 0 \}  $$
and $R/P = \F_2$. Let us now show that
 $P^5=0$. The proof
of this fact goes as the one given in \cite{Livne}, except that we
consider now products of $5$ elements $x_1, ..., x_5$ in $P$
(instead of $4$ elements): assuming that their product is non-zero
we derive a contradiction, as in \cite{Livne}, by considering the
$\F_2$-algebra
$$ R_0 = \F_2[T_1, ..., T_5]/(T_1^2, ... , T_5^2) $$ and the injective map of
$\F_2$-algebras from $R_0$ to $R$ that sends $T_i$ to $x_i$. The
contradiction follows from the inequalities: $\dim_{\F_2} R_0 = 32 >
16 \geq \dim_{\F_2} R$.\\
Using the fact that $P^5 = 0$, the rest of the proof of the
assertion follows as in \cite{Livne}, changing ``cubic polynomial"
by ``quartic polynomial", and ``non-cubic" by ``non-quartic". \\
The assertion being proved, we conclude as in \cite{Livne} from
Nakayama's lemma that the $\Z_2$-span of $\{ \rho(\sigma) \; | \;
\sigma \in \Sigma \cup \{1 \} \}$ is all of $M$  (recall that $M$ is
the
linear $\Z_2$-span of $\rho(G)$).\\
Here comes the only place where equality of the traces over $\Sigma$
is used in Livn\'e's proof (cf. \cite{Livne}): he considers the map:
$\alpha: M \rightarrow \Om$ defined by $\alpha(a,b) = \trace a -
\trace b$. Since this map is $\Om$-linear and $\alpha(I,I)=0$,
assuming that $\alpha(\rho(\sigma))=0$ for every $\sigma \in \Sigma$
Livn\'e concludes that $\alpha = 0$, i.e., equality of the traces of
$\rho_1$ and $\rho_2$. We can argue in the same way but using the
map: $ \beta: M \rightarrow \Z_2$ given by
$$\beta (a,b) = \trace a + (\trace a)^\phi - \trace b - (\trace
b)^\phi $$ where $\phi$ is the order two element in $\Gal(E/ \Q_2)$.
Observe that when applied to numbers in $F$, $\phi$ agrees with the
order two element $\tau$ in $\Gal(F/ \Q)$.\\
\\ Observe that if $a = \rho_1(g)$, then $(\trace a)^\phi = (\trace
\rho_1(g))^\phi = \trace (\rho'_1 (g))$, and similarly for $b =
\rho_2(g)$. This map is  $\Z_2$-linear and satisfies $\beta(I,I)=0$.
Then, assumption (2) (ii) in the theorem implies that for elements
of $G$ in $\Sigma$ (recall that these elements correspond to
Frobenius elements for primes in $T$) the map $\beta$ vanishes, thus
we conclude (as Livn\'e does for $\alpha$) that $\beta = 0$, which is
the equality of the traces of two $4$-dimensional $2$-adic Galois
representations. Applying Brauer-Nesbitt we conclude that these
$4$-dimensional Galois representations have isomorphic
semi-simplifications.
\end{proof}

\subsection{Proof of Theorem \ref{main-theorem}}

We want to apply Theorem \ref{Livne} to the $2$-adic Galois
representations $\sigma_2, \sigma_{\ff,2}$ over $F=K$.  In Proposition
\ref{Prop:4} we proved that $\sigma_2$ has even traces by geometric
considerations, and in Section \ref{s:image} we proved that the same
is true for the representation $\sigma_{\ff,2}$.
Recall from the introduction that both residual representations $\bar\sigma_2, \bar\sigma_{\ff,2}$
have  image in $\SL(\F_4)$.
In particular their determinants are congruent modulo $2$.
Thus the first
hypothesis of Theorem \ref{Livne} is satisfied. The field $K_S$ of
Theorem \ref{Livne} is the compositum of the thirty one field
extensions computed in the proof of Proposition \ref{prop:mod}. The
set of primes in $K$ in the set

\begin{multline*}
T=\{\langle 61,26-\sqrt{5}\rangle,
\langle 59,\sqrt{5}+8\rangle,
\langle 149,\sqrt{5}-68\rangle,
\langle 211,\sqrt{5}+65\rangle,
\langle 101, \sqrt{5}-45 \rangle,\\
\langle 19, \sqrt{5} +9\rangle,
\langle 229, \sqrt{5} - 66 \rangle,
\langle 11, \sqrt{5} - 4 \rangle,
\langle 11, \sqrt{5} + 4 \rangle,
\langle 109, \sqrt{5} - 21 \rangle,\\
\langle 19, \sqrt{5} - 9 \rangle,
\langle 701, \sqrt{5} - 53 \rangle,
\langle 211, \sqrt{5} - 65 \rangle,
\langle 29, \sqrt{5} - 11 \rangle,
\langle 59, \sqrt{5} - 8 \rangle,\\
\langle 181, \sqrt{5} - 27 \rangle,
\langle 239, \sqrt{5} + 31 \rangle,
\langle 31, \sqrt{5} + 6 \rangle,
\langle 79, \sqrt{5} - 20 \rangle,
\langle 71, \sqrt{5} - 17 \rangle,
13, \\
\langle 401, \sqrt{5} - 178 \rangle,
\langle 449, \sqrt{5} - 118 \rangle,
\langle 241, \sqrt{5} - 103 \rangle,
\langle 89, \sqrt{5} - 19 \rangle,
7,\\
\langle 79, \sqrt{5} + 20 \rangle,
\langle 239, \sqrt{5} - 31 \rangle,
\langle 41, \sqrt{5} - 13 \rangle,
\langle 31, \sqrt{5} - 6 \rangle,
\langle 71, \sqrt{5} + 17 \rangle\},
\end{multline*}
saturate the set $\Gal(K_S/K)\backslash \{0\}$. They are ordered in
such a way that the primes (under class field theory) correspond to
the extensions listed in Table \ref{table:extensions} in the same
order.

Thus $T$ saturates the set $\Gal(K_S/K)\backslash \{0\}$, 
but we can still eliminate
some prime ideals of big norm by replacing $T$ by a non-quartic test set by Lemma \ref{lemma:non-cubic}. 
To do so, we fix a standard basis of $\F_2^5$ corresponding to the following quadratic extensions of $F$:
\[
x^2-3\frac{\sqrt{5}-5}2, x^2-3, x^2+\frac{\sqrt{5}+5}2,x^2-3\sqrt{5},x^2-2.
\]
In this basis,
the elements corresponding to the primes above $701$,
$449$ and $401$ correspond to the elements $(1,1,0,0,1)$,
$(0,1,1,1,0)$ and $(1,1,0,1,0)$ respectively in $\F_2^5$.

We claim that the set $T'$ obtained from $T$ by removing these three
elements is a non-quartic set in $\F_2^5$. 
To see this, we use
Lemma \ref{lemma:non-cubic} with the fact that $T'\cup\{0\}$ contains
the four  hyperplanes
\[
\begin{cases}
x_2=0\\
x_4+x_5=0\\
x_1+x_3=0\\
x_1+x_2+x_3+x_4+x_5=0
\end{cases}
\]
and the extra point $(0,1,1,0,1)$ (which corresponds to one of the primes above $59$).


\smallskip

It was checked in
\cite{consani} that
the characteristic polynomials of the $4$-dimensional Galois
representations $\sigma_2 \oplus \sigma'_2$ and $\sigma_{\ff,2}
\oplus \sigma_{\tau(\ff),2}$
agree at the primes of $K$ above rational primes smaller than $100$.
For the remaining set of primes above the set of rational primes
\[
\{101, 109, 149, 181, 211, 229, 239, 241\},
\]
the same was checked by us, see Appendix \ref{points} and the table at \cite{table}.
By Theorem \ref{Livne} we conclude that the two $4$-dimensional
Galois representations have indeed isomorphic semi-simplifications.
In particular, any irreducible component of one of them must be
isomorphic to some irreducible component of the other, thus since we
know that $\sigma_{\ff,2}$ and  $\sigma_{\tau(\ff),2}$ are
irreducible, one of them must be isomorphic to $\sigma_2$ (and the
other to $\sigma'_2$). This proves Theorem \ref{main-theorem}. \qed

\appendix

\section{Appendix: Counting points}
\label{points}

In this appendix we indicate how to count the number of points on the
Consani-Scholten quintic $\tilde X$ over finite fields.  In
particular, we give the traces of the Galois representations
$\sigma_\lambda, \sigma_\lambda'$ at the primes $p>100$ needed to
prove Theorem \ref{main-theorem}.  For most part, we follow the
approach from \cite{consani}.

Recall that $\tilde X$ is given affinely by the symmetric equation in the Chebyshev polynomial $P_5(y,z)$:
\[
X= \{P_5(x_1,x_2)=P_5(x_4,x_5)\}\subset\mathbb{A}^4.
\]
Thus we can count the number of points of $\tilde X$ over some finite field $\F_q$ as follows:
\begin{enumerate}
\item
Compute the affine number of points $\#X(\F_q)$.
\item
For the projective closure $\bar X\subset\PP^4$,
let $S\subset \PP^3$ be $\bar X- X$.
Compute $\#S(\F_q)$.
\item
Compute the contribution from the exceptional divisors in the resolution $\tilde X\to \bar X$.
\end{enumerate}

For (1), we can proceed by counting how often each value in $\F_q$ is attained by the Chebyshev polynomial $P_5(y,z)$ over $\mathbb{A}^2$.
Due to the symmetry, $\# X(\F_q)$ is the sum of the squares of these numbers.

For (2), 
note that $S$ is the Fermat quintic surface given by the model
\[
S=\{x_1^5+x_2^5=x_4^5+x_5^5\}\subset\PP^3.
\]
In \cite{consani} it was pointed out that for $q\not\equiv 1\mod 5$ one has $\# S(\F_q)=1+q+q^2$.
Meanwhile for $q\equiv 1\mod 5$, we could either use the zeta function of $S$ and its description in terms of Jacobi symbols due to A.~Weil or proceed along the same lines as above, i.e.~count points over $\mathbb A^4$ using symmetry
and then take into account that we are actually working over $\PP^3$ (substract $1$ and divide by $q-1$).
However either approach would a priori impose the same complexity as for computing $\#X(\F_q)$.
Luckily counting $\{(y,z)\in\mathbb{A}_{\F_q}^2; y^5+z^5=a\}$ can be improved by noting that scalar multiplication acts as multiplication by fifth powers on the values.
Hence $\#\{(y,z)\in\mathbb{A}_{\F_q}^2; y^5+z^5=0\}=(q-1)\cdot \#\{[y,z]\in\PP_{\F_q}^1; y^5+z^5=0\}+1$ and for any $a\neq 0$ with  $O(a)=a\cdot(\F_q^*)^5$ denoting the $a$-orbit under multiplication by fifth powers in $\F_q^*$:
\[
\#\{(y,z)\in\mathbb{A}_{\F_q}^2; y^5+z^5=a\} = 5\sum_{o\in O(a)} \#\{[y,z]\in\PP_{\F_q}^1; y^5+z^5=o\}.
\]
(Strictly speaking the sets on the right are ambiguous, but scalar multiples end up in the same orbit, so the contribution does not depend on the chosen representative $[y,z]\in\PP^1$.)
Essentially this simplification reduces the algorithm's running time from $q^2$ to $q$ compared with computing $\#X(\F_q)$.

Finally for (3) we recall from 3.3 that the 120 nodes are always defined over the extension of $\F_q$ containing the $15$th roots of unity, and that their rulings are always defined over the same field.
So if a node is defined over $\F_q$, then its exceptional divisor contributes $q^2+2q$ additional points.
Thus we find $\#\tilde X(\F_q)$.
This allows us to compute the trace 
\[
t_q=\text{trace Frob}_q^*(H^3(\tilde X_{\bar\Q},\Q_\ell))
\]
 through the Lefschetz fixed point formula \eqref{eq:trace}.
Here we do not need to know $h_q$ in advance since it is determined (for $q>20$ and $b_3(\tilde X)=4$) by the inequality
\[
\mid t_q\mid \leq 4q^{3/2}.
\]
We obtain the characteristic polynomial of Frob$_q^*$ on $H^3(\tilde X_{\bar\Q}, \Q_\ell)$:
\[
\mu_q(T) = T^4 - t_q T^3 + \frac 12 (t_q^2-t_{q^2}) T^2 - t_qq^3T+q^6.
\]
In the present situation, we know that $L_q(T)$ will always split over $\Q(\sqrt{5})$:
\[
\mu_q(T) = (T^2 - \alpha_q T + q^3)(T^2 -\alpha_q^\sigma T + q^3), \;\;\; \alpha_q\in\Q(\sqrt{5}).
\] 
The traces $\alpha_p, \alpha_p^\sigma$ appear as eigenvalues of the
Hilbert modular form.  In the following table, we collect one of the
traces together with the numbers of points of $X$ and $S$ over $\F_p$
and $\F_{p^2}$ for all primes $p>100$ needed to prove Theorem
\ref{main-theorem}.

\begin{table}[ht!]
\[
{\small 
\begin{array}{c|r|r|r|r|c}
p & \#X(\F_p) & \# X(\F_{p^2})\;\;\;\;\, & \# S(\F_p) & \# S(\F_{p^2}) \;& \alpha_p\\
\hline
101& 1222681& 1063601210405 & 14655& 104338955 & -598-476  \sqrt{5}\\
109& 1338593& 1679922873825 & 11991& 141787855 & 890+468  \sqrt{5}\\

149& 3395857& 10952392903505 &  22351& 494061055 & 150-344  \sqrt{5}\\
181& 6562145& 35183310464645 & 39455& 1074841355 & -898-288  \sqrt{5}\\

211& 10261235& 88285583898085 & 49205& 1984280555 & -1228-1616  \sqrt{5}\\

229& 12214593& 144270849112465 & 52671& 2752837855 & -210+940  \sqrt{5}\\ 
239& 13872967& 186440164574105 & 57361& 3265836055 & 3240+944  \sqrt{5}\\
241 & 15137985 & 195998061709305 & 65255& 3375066455 & -4938+172  \sqrt{5}\\
\end{array}}
\]
\end{table}

\section{Appendix: Sturm Bound (by Jos\'e Burgos Gil and Ariel Pacetti)}

\label{appendix:sturm}
The aim of this appendix is to show how a Sturm bound can be obtained
for the modular form of level $6\sqrt{5}$ and parallel weight $10$. We
expect to extend the result to any real quadratic fields in
a future work. Following the previous notation, $F$ will denote the
real quadratic field $\Q(\sqrt{5})$.


\subsection{Desingularization and a Hecke bound over $\CC$}
\label{sec:desing-sturm-bound}

Let $H(3)$ be the Hilbert modular surface obtained as the quotient of
the product of two copies of the Poincare upper half plane modulo the
action of the congruence group
\[
\Gamma(3):= \left\{
\left(\begin{array}{cc}
\alpha & \beta\\
\gamma & \delta \end{array} \right) \in \SL_2(\Om_F) \; : \; \alpha \equiv \delta \equiv 1 \pmod{3}, \beta, \gamma \in 3\Om_F\right \}.
\]

The group $\Gamma(3)$ has no fixed elliptic points for this action
(see \cite{VDG} page 109); it has $10$ non-equivalent cusps. Let
$\overline H(3)$ be the minimal 
compactification of $H(3)$ obtained by adding one point for each cusp.
The
surface $\overline H(3)$ is singular at all such points. Denoting by
$\widetilde{H}(3)$ the 
minimal desingularization of $\overline H(3)$, we get that the diagram of the
desingularization of $\overline H(3)$ at any cusp is the following
(see \cite{VDG}, page 193):

\begin{figure}[h]
    \begin{center}
        \includegraphics[width=5cm]{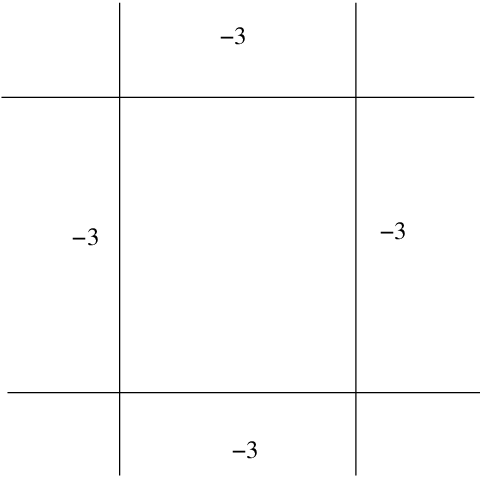}
    \end{center}
\end{figure}

Denote by $c_i$, $1 \le i \le 10$ the different cusps (where $c_1$ is
the cusp at infinity) and denote by $S_i$, $1 \le i \le 10$ the
exceptional divisor at the $i$-th cusp. The surface $\widetilde{H}(3)$ is of
general type (by Theorem 3.4 of \cite{VDG}).

We want a criterion to show that a Hilbert modular form whose Fourier
expansion starts with many zeroes is actually the zero modular
form. This is a generalization of the Sturm bound to Hilbert modular
forms.  To this end we need a nef (numerically eventually free) divisor.
Let $F_1$ be the curve defined in \cite{VDG} page 88.  It has $30$
disjoint connected components made of curves with self-intersection
number $-2$ and it meets each connected component of the
desingularization at the cusps in three points (see \cite{VDG}, page
193).

\begin{lemma} 
\label{lemma:intnum}
The intersection numbers between the curves $S_{i}$ and $F_{1}$ are:
\begin{itemize}
\item $S_i \cdot S_j = \begin{cases} 
\ 0 & \text{ if } i \neq j,\\
-4 & \text{ if } i = j.\end{cases}$
\item $F_1 \cdot F_1 = -60$.
\item $S_i \cdot F_1 = 12$.
\end{itemize}
\end{lemma}

Consider the divisor
\[
D':= \frac{1}{5} \left ( \sum_{i=1}^{10} S_i + 2F_1\right ).
\]

\begin{lemma} 
\label{lemma:canonical}
The divisor $D'$ is nef and it agrees with the canonical divisor. Its
self-intersection number is given by
\[
D'\cdot D'=8.
\]
\end{lemma}

\begin{proof} 
In example $7.5$ of \cite{VDG} (page 179) it is proved that $D'$ equals
the canonical divisor and its self-intersection number is
computed. The fact that it is nef follows from the fact that
$\widetilde{H}(3)$ is a minimal surface of general type.
\end{proof}

\begin{remark}
Since $\widetilde{H}(3)$ is a complex smooth projective surface, there
is a unique canonical divisor (which equals the dualizing sheaf).
\end{remark}

Let $M_{2k}(\Gamma(3))$ denote the vector space of modular forms of
parallel weight $2k$ for $\Gamma(3)$. It is the space of global
sections of the line
bundle $\mathcal{O}(k (D' + S))=\mathcal{O}(D' + S)^{\otimes k}$:
\[
M_{2k}(\Gamma(3)) =\Gamma (\mathcal{O}(k (D' + S))),
\]
where $S=\sum_{i=1}^{10} S_i$. Similarly, the space of cusp forms
$S_{2k}(\Gamma(3))$ is given by the divisor $k(D'+S)-S$.  

We want to add some vanishing conditions to $M_{2k}(\Gamma(3))$ such
that the space of forms with these vanishing conditions is empty. Let
$a$ be a positive 
integer, and $G$ be a Hilbert modular form. We say that $G$ vanishes with
order $a$ at the cusp $c_i$ if $G$ is a section of
$\mathcal{O}(k(D'+S)-aS_i)\subset \mathcal{O}(k(D'+S))$. 

Let $G$ be a form which vanishes with order $a$ at all the cusps and
with order 
$a+b$ at the infinity cusp, i.e. it belongs to the space given by the
divisor
\[
E = k(D'+S) -aS -b S_1.
\]
It follows from Lemma~\ref{lemma:intnum} and
Lemma~\ref{lemma:canonical} that $D'\cdot S_i=4$ for all $1\le i \le
10$, so
\[
E \cdot D'=k(D'\cdot D'+S \cdot D')-aS\cdot D' -b S_{1}\cdot D'=48k-40a-4b.
\]
If $b+10a > 12k$, this intersection number is negative and, since $D'$
is nef, the space
of global sections of $\mathcal{O}(E)$ is the zero space. This implies

\begin{thm}
\label{thm:sturm1}
If $G$ is a Hilbert modular form of parallel weight $2k$ for
$\Gamma(3)$ which vanishes with order $a$ at all cusps and with order
$a+b$ at the infinity cusp and $b+10a>12k$, then $G$ is the zero form.
\end{thm}

\begin{coro}
If $G$ is a Hilbert modular form of weight $(k_1,k_2)$, with $k_1
\equiv k_2 \pmod 2$, for $\Gamma(3)$ which vanishes with order $a$ at
all cusps and with order $a+b$ at the infinity cusp and
$b+10a>3(k_1+k_2)$, then $G$ is the zero form.
\end{coro}
\begin{proof}
Just apply the previous Theorem to the form $G(z_1,z_2)\cdot G(z_2,z_1)$.
\end{proof}

To relate the order of vanishing of a modular form at a cusp with the
$q$-expansion we need to compute explicitly the first step of the
desingularization of the cusp. In the case of the infinity cusp, this
implies computing the local ring of the cusp, which is done
in \cite{VDG} Chapter II, Section 2. The stabilizer of the infinity
cusp for $\Gamma(3)$ is given by
\[
\left \{\left(\begin{array}{cc}
\epsilon & \alpha\\
0 & \epsilon^{-1} \end{array} \right) \; : \; \alpha \in 3\Om_F \text{
and }\epsilon \equiv 1 \pmod{3}\right\},
\]
i.e. it is of type $(M,V)=(3 \Om_F, U_F^8)$, where $U_F^8
= \langle \frac{47}{2} + \frac{3}{2}\sqrt{5} \rangle$. The dual of $M$
is given by $M^\vee = \frac{\Om_F^\vee}{3}= \frac{\Om_F}{3\sqrt{5}}$,
so any Hilbert modular form for $\Gamma(3)$ has a $q$-expansion at the
infinity cusp of the form
\[
\sum_{\substack{\xi \gg 0\\ \xi \in \frac{\Om_F}{3\sqrt{5}}}} a_\xi
\exp(\xi z_1 + \tau(\xi) z_2).
\]

Let $M_+$ denote the elements of $M$ which are totally positive, and
consider the embedding of $M_+$ in $(\RR_+)^2$, given by
\[
\mu \mapsto (\mu,\tau \mu).
\]
Denote by $A_k=(A_k^1,A_k^2)$, $k \in \Z$ the vertices of the boundary
of the convex hull of the image of $M_+$, ordered with the condition
$A_{k+1}^1 < A_k^1$ for all $k$. Any pair $(A_{k-1},A_k)$ is a basis
for $M$ as $\ZZ$-module (see \cite{VDG} Lemma 2.1).
This determines
an isomorphism 
\[
M \backslash \CC^2 \to \CC^\times \times \CC^\times,
\]
which maps $z=(z_{1},z_{2})$ to $(u_{k-1},u_k)$, where
\begin{equation}
  \label{eq:1}
  \exp(z_j) = u_{k-1}^{A_{k-1}^j} u_{k}^{A_{k}^j}, \quad \text{ for }j=1,2.  
\end{equation}
Let $\sigma_k$ denote the cone spanned by $A_{k-1}$ and $A_k$, i.e.
\[
\sigma_k=\{s A_{k-1} + tA_k \; : \; s,t \in \R_+\}.
\]
The desingularization of the infinity cusp is obtained by taking a
copy of $\CC^2$ for each element $\sigma_k$ and gluing them together
in terms of the change of basis matrix
(see \cite{VDG} page 31). 

Let $\xi \in M^\vee$ be a totally positive element. Then in the copy
corresponding to $\sigma_k$,
\begin{eqnarray*}
\exp(\trace(\xi z)) &=& \exp(\xi z_1 + \tau(\xi)z_2) = u_{k-1}^{\xi
  A_{k-1}^1+\tau(\xi) A_{k-1}^2}\cdot u_k^{\xi
  A_{k}^1+\tau(\xi)A_{k}^2}=\\ & =& u_{k-1}^{\trace(\xi A_{k-1})} \cdot
u_{k}^{\trace(\xi A_{k})}.
\end{eqnarray*}
Now, we denote by $L_k$, the component of the exceptional divisor
$S_{1}$, over
the infinity cusp, that corresponds to the ray through $A_k$. Observe
that $L_{k+4}=L_{k}$. We have:

\begin{prop}
Let $G$ be a Hilbert modular form with $q$-expansion at infinity
\[
G(z_1,z_2)=\sum_{\substack{\xi \gg 0\\ \xi \in \frac{\Om_F}{3\sqrt{5}}}}
a_\xi \exp(\xi z_1 + \tau(\xi) z_2).
\]
Let $L_k$ be as above. Then $\ord_{L_k}(G) > K$ if and only if $a_\xi=0$
for all $\xi \in M^\vee$, $\xi\gg 0$ with $\trace(\xi A_k)\le K$.
\label{prop:fourier}
\end{prop}

\begin{table}[h]
\begin{center}
\begin{tabular}{|c|c|}
\hline
Name & Point\\
\hline
$A_0$ & $3(1,1)$\\
$A_1$ & $3(1+\omega,1+\tau(\omega))$\\
$A_2$ & $3(2+3\omega,2+3\tau(\omega))$\\
$A_3$ & $3(5+8\omega,5+8\tau(\omega))$\\
\hline
\end{tabular}
\end{center}
\caption{First boundary points}
\label{BP}
\end{table}

In Table \ref{BP} it is shown a set of nonequivalent boundary points
of the convex  hull of $M_{+}$, where $\omega$ denotes the element
$\frac{1+\sqrt{5}}{2}$. It is clear that they differ by powers of
$\omega^2$, and since the matrix
$\left(\begin{smallmatrix}\omega^2&0\\0&\omega^{-2}\end{smallmatrix}\right)
\in \Gamma_0(3)$, a Hilbert modular form for $\Gamma_0(3)$ will vanish
with the same order in the four components $L_{k}$, $k=0,\dots,3$, of
$S_{1}$. The vanishing condition corresponding to
$L_0$, reads  
\[
\ord_{L_0}(\exp(\xi z))=3\trace(\xi).
\]
In particular, a modular form vanishes at the cusp if and only if
$a_0=0$. 

The above discussion implies the following theorem for $\Gamma_0(3)$.

\begin{thm} 
Let $G$ be a Hilbert modular form of parallel weight $2k$ for
$\Gamma_{0}(3)$ which vanishes with order $a$ at all cusps and whose
Fourier expansion at the infinity cusp is given by
\[
G = \sum_{\substack{\xi \gg 0\\ \xi \in \frac{\Om_F}{3\sqrt{5}}}}
a_\xi \exp(\xi z_1 + \tau(\xi) z_2).
\]
If $a_\xi = 0$ for all $\xi$ with $\trace(\xi) \le 4k-3a$ then $G$
is the zero form.
\label{thm:Sturm2}
\end{thm}
\begin{proof}
  If $a_\xi = 0$ for all $\xi$ with $\trace(\xi) \le 4k-3a$, by
  Proposition \ref{prop:fourier}, $G$ vanishes with order greater than  
  $12k-9a$ at the infinity cusp.  Thus the result follows from Theorem
  \ref{thm:sturm1}. 
\end{proof}

\begin{coro}
Let $G$ be a Hilbert modular form of weight $(k_1,k_2)$, with
$k_1\equiv k_2 \pmod 2$, for $\Gamma_{0}(3)$ which vanishes with order
$a$ at all cusps and whose Fourier expansion at the infinity cusp is
given by
\[
G = \sum_{\substack{\xi \gg 0\\ \xi \in \frac{\Om_F}{3\sqrt{5}}}}
a_\xi \exp(\xi z_1 + \tau(\xi) z_2). 
\]
If $a_\xi = 0$ for all $\xi$ with $\trace(\xi) \le
(k_1+k_2)-3a$ then $G$ is the zero form.
\end{coro}

\subsection{Moduli interpretation and integral models}

In order to make the computation of the previous section work over
finite fields, we need to use the integral structure of the modular
Hilbert surface and of the modular curve $X(3)$. It comes from their
moduli interpretation. Let us follow the notation of \cite{Burgos}.

We fix $\zeta_3$ a
third-root of unity and we denote by $\delta =(\sqrt {5})^{-1}$
the different of $F$.

An abelian scheme $A \rightarrow S$ of
relative dimension $2$, together with a ring homomorphism
\[
\iota: \Om_F \rightarrow \End(A)
\]
is called an abelian surface with multiplication by $\Om_F$, and is
denoted by the pair $(A,\iota)$. This gives an $\Om_F$-multiplication
in the dual abelian surface $A^\vee$. An element $\mu \in
\Hom(A,A^\vee)$ is called $\Om_F$-linear if $\mu \iota(\alpha) =
\iota(\alpha)^\vee \mu$ for all $\alpha \in \Om_F$. Denote by
$\bfP(A)$ the sheaf for the \'etale topology on $\Sch/S$ defined by 
\[
\bfP(A)_T = \{\lambda: A_T \rightarrow A_T^\vee \; : \; \lambda \text{
  is symmetric and $\Om_F$-linear}\},
\]
for all $T \rightarrow S$. The subsheaf $\bfP(A)^+$ is the subsheaf of
polarizations in $\bfP(A)$. The pair $(A,\iota)$ is said to satisfy
the Deligne-Pappas condition, denoted by (DP), if the canonical
morphism of sheaves 
\[
A \otimes_{\Om_F} \bfP(A) \mapsto A^\vee
\]
is an isomorphism. In this case, $\bfP(A)$ is a locally constant sheaf
of projective $\Om_F$-modules of rank $1$. 

Since the class number of $\Om_F$ is one, we can restrict to consider
only $\Om_F$-\emph{polarizations}. An $\Om_F$-polarization on a pair
$(A,\iota)$ is a morphism of $\Om_F$-modules $\psi:\Om_F \rightarrow
\bfP(A)_S$ taking $\Om_F^+$ to $\bfP(A)^+$ such that the natural
homomorphism
\[
A \otimes_{\Om_F} \Om_F \rightarrow A^\vee, \quad a \otimes
\alpha \mapsto \psi(\alpha)(a)
\]
is an isomorphism. 

Suppose $S$ is a scheme over $\Spec \ZZ[1/3]$. A level $3$-structure
on an abelian surface $A$ over $S$ with real multiplication by $\Om_F$
is an $\Om_F$-linear isomorphism
\[
\varphi:(\Om_F/3)^2_S \rightarrow A[3] 
\]
between the constant group scheme defined by $(\Om_F/3)^2$ and the
$3$-torsion of $A$. 

\begin{thm}
The moduli problem ``Abelian surfaces over $S$ with real
multiplication by $\Om_F$ satisfying (DP) condition, $\Om_F$-polarization
and level 
$3$-structure'' is represented by a regular algebraic scheme
${\mathcal H}(3)$
which is flat and of relative dimension two over $\Spec
\ZZ[1/3,\xi_3]$. Furthermore, it is smooth over $\Spec
\ZZ[1/15,\xi_3]$.
\end{thm}

\begin{proof}
See \cite{Go} Theorem 2.17, p. 57;Lemma 5.5, p. 99.
\end{proof}

\begin{remark}
The scheme ${\mathcal H}(3)$ is not geometrically irreducible, it has
$\# (\Om_F/3)^\times=8$ connected components over $\bar{\Q}$. In fact,
the 8 components are defined over $\Spec \ZZ[1/15,\xi_3]$. This
definition is not the same as the one given in \cite{Rapoport} (which
is connected), it is a topological cover of degree $4$ of it. Let $S$
be a scheme over $\ZZ[1/3]$, then the abelian scheme $A$ has a Weil
pairing $e_3:A[3] \times A^\vee[3] \to \mu_3$, which satisfies
$e_3(\alpha a,b) = e_3(a,\alpha b)$. There exists an $\Om_F$-bilinear
form $e_{\Om_F}:A[3]\times A^\vee[3] \to
(\delta^{-1}/3\delta^{-1})(1)$ such that $e_3 =
\Trace(e_{\Om_F})$. Any element in $\lambda \in \bfP(A)_T$ defines a
homomorphism between $A$ and $A^\vee$ which is trivial on $A[3]$ if
and only if $\lambda \in 3 \bfP(A)_T$ for any morphism $T \rightarrow
S$. Since $e_{\Om_F}$ is non-degenerate, $\bfP(A)\otimes_{\Om_F}
\Lambda^2_{\Om_F}A[3] = \delta^{-1}/3 \delta^{-1}(1)$.

Any element $\phi \in \Isom(\mu_3,\ZZ/3)$ gives an isomorphism between
$\delta^{-1}/3 \delta^{-1}(1)$ and $\delta^{-1}/3 \delta^{-1}$. In
\cite{Rapoport} the only level $3$-structures considered are the
$\varphi$ such that in the following diagram
\[
\xymatrix{
\bfP(A)\otimes_{\Om_F}\Lambda^2_{\Om_F}A[3]
\ar@{->}[d]^{\sim}_{e_{\Om_F} \otimes \psi}
\ar@{=}[r] &  \delta^{-1}/3\delta^{-1}(1) \ar@{..>}[d] \\
\delta^{-1} \otimes_{\Om_F} \Lambda^2_{\Om_F}(\Om_F/3) \ar@{=}[r] &
\delta^{-1}/3\delta^{-1}
}
\]
the vertical dotted arrow is given by an element $\phi \in
\Isom(\mu_3,\ZZ/3)$. Since all such maps differ by (multiplication by)
an element in $(\Om_F/3)^\times$, the two assertions follow.
\end{remark}


\begin{remark}
The group $\GL_2(\Om_F/3)$ acts on ${\mathcal H}(3)$, where an element
$M$ send a level $3$-structure $\varphi$ to the level $3$-structure
$\varphi \circ M$. The subgroup $\SL_2(\Om_F/3)$ acts on each
connected component of ${\mathcal H}(3)\otimes \bar{\Q}$, while the
subgroup $H_{F}=\{\left( \begin{smallmatrix} \alpha & 0\\0 &
  1\end{smallmatrix} \right) \text{ such that }\linebreak \alpha \in
  (\Om_F/3)^\times\}$ acts transitively on the set of connected
  components.
\end{remark}

\begin{thm} \label{thm:3}
There is a toroidal compactification $h_3:\widetilde{{\mathcal
 H}}(3) \to \ZZ[\zeta_3,1/3]$ of ${\mathcal H}(3)$ that
is smooth at infinity. The complement $\widetilde{{\mathcal H}}(3)
\backslash {\mathcal H}(3)$ is a relative divisor with normal crossings.
\end{thm}

\begin{proof}
See \cite{Chai} Theorem 3.6, Theorem 4.3 and \cite{Rapoport} Theorem
5.1 and Theorem 6.7.
\end{proof}

The set of complex points of $\mathcal{H}(3)$ is equal to 8 copies of
the surface $H(3)$ considered in the previous section, while the set
of complex points of $\widetilde {\mathcal{H}}(3)$ is equal to 8
copies of $\widetilde{H}(3)$. 

If we study the moduli problem for $1$-dimensional abelian varieties
(i.e. elliptic curves), we have the advantage that they are already
principally polarized. As in the two-dimensional case, if $S$ is a
scheme over $\ZZ[1/3]$, a level $3$-structure on an elliptic curve $E$
over 
$S$ is a $\ZZ$-linear isomorphism
\[
\varphi:(\ZZ/3)_{S}^2 \rightarrow E[3]
\]
between the constant group scheme defined by $(\ZZ/3)^2$ and the
$3$-torsion of $E$.

\begin{thm} \label{thm:2}
The moduli problem ``elliptic curves over $S$ with level
$3$-structure'' is represented by a smooth affine curve
$\mathcal{Y}(3)$ over 
$\ZZ[1/3]$. Furthermore, the
category ${\mathcal M}_3[1/3]$ of ``generalized elliptic curves over
$S$, that have smooth generic fibres, singular fibres  whose Neron
polygons have 
$3$-sides and with a level $3$-structure'' is a projective smooth scheme
$\mathcal{X}(3)$ over $\ZZ[1/3]$.
\end{thm}
\begin{proof} See \cite{Deligne}, Corollary 2.9. 
\end{proof}
\begin{remark}
  The group $\GL_2(\ZZ/3)$ acts on $\mathcal{X}(3)$, in the same way
  as in the two dimensional case, i.e. an element $M \in \GL_2(\ZZ/3)$
  sends the level $3$-structure $\varphi$ to the level $3$-structure
  $\varphi \circ M$. The subgroup $\SL_2(\ZZ/3)$ acts on each
  connected component of $\mathcal{X}(3) \otimes \bar{\Q}$ while the
  subgroup $H _{\Q}=\{ \left( \begin{smallmatrix} \alpha & 0 \\ 0 &
      1\end{smallmatrix} \right) \text{ such that } \alpha \in
  (\ZZ/3)^\times\}$ acts transitively in the set of connected
  components.
\end{remark}

We want to study the inclusions of $\mathcal{Y}(3)$ into ${\mathcal H}(3)$. If
$E$ is an elliptic curve over $S$, the abelian variety $A_E=E
\otimes_\ZZ \Om_F$ has a canonical $\Om_F$-action $\iota_E:\Om_F
\rightarrow \End(A_E)$. Furthermore,
\[
A_E \simeq E \times_S E,
\]
where the isomorphism depends on chosing a basis for $\Om_F$ as
$\ZZ$-module. The dual abelian variety $A_E^\vee$ is isomorphic to $E
\otimes_\ZZ \delta^{-1}$. Furthermore, $\bfP(A_E) \simeq \delta^{-1}
\simeq \Om_F$ (i.e. $A_E$ has a canonical principal polarization
$\psi_E$), where the isomorphism preserves positivity and the
Deligne-Pappas condition holds (see \cite{Burgos} Lemma 5.10). Let
$(E,\varphi)$ be an elliptic curve with level $3$-structure. 

Consider the natural inclusion $\SL_2(\Z/3) \hookrightarrow
\SL_2(\Om_F/3)$, and let ${\mathcal N}=\{N_i\}_{i=1}^{30}$ be a set
  of representatives for the quotient set $\SL_2(\Om_F/3) /
  \SL_2(\ZZ/3)$.  Each element $N$ of ${\mathcal N}$ gives rise to an
  embedding of $\mathcal{Y}(3)$ into ${\mathcal H}(3)$, which associates to the
  pair $(E,\varphi)$ the element $(A_E,\iota_E,\psi_E, \varphi \circ
  N)$.
  For shortness we abuse notation by writing $\varphi \circ N$. The map
  $\varphi$ extends by $\Om_F$-linearity to a map
\[
\widetilde{\varphi}:(\ZZ/3)^2 \otimes_\ZZ \Om_F = (\Om_F/3)^2\rightarrow A_E[3],
\]
and the element $N$ acts in $(\Om_F/3)^2$. By choosing a
connected component of $\mathcal{Y}(3) \otimes \bar{\QQ}$ and of ${\mathcal
  H}\otimes \bar{\QQ}$, the $30$ embeddings obtained from the set
${\mathcal N}$ (up to composition with an element of $H_{F}$ if
necessary) give us open dense subsets of the $30$ connected components of
the curve $F_1$ 
from the previous section.

\begin{thm}\label{thm:1}
The closed inmersion $\mathcal{Y}(3) \hookrightarrow {\mathcal H}(3)$ of schemes
over $\ZZ[\zeta_3,1/15]$ extends to a closed inmersion $\mathcal{X}(3)
\hookrightarrow \widetilde{\mathcal H}(3)$.
\end{thm}

\begin{proof}
Following the referee's advice, we omit the proof of this Theorem.
\end{proof}

Theorem \ref{thm:1} has the following direct consequences.  Let
$\widetilde {\mathcal{H}}$ be one of the irreducible components of
$\widetilde {\mathcal{H}}(3)$ over $\Spec \ZZ[1/15,\xi_3]$. The set of
complex points of $\widetilde {\mathcal{H}}$ agrees with the surface
$\widetilde H(3)$ of section \ref{sec:desing-sturm-bound}. Let $Z$ be
any irreducible component of the divisor $D'$ of $\widetilde H(3)$
introduced in the same section. Then $Z$ is defined over
$\QQ(\zeta_{3})$. Let $\mathcal{Z}=\overline {Z}$ be the Zariski
closure of $Z$ in $\widetilde {\mathcal{H}}$.

\begin{coro} \label{cor:irredfib}
  For every prime $\id{p}$ of $\ZZ[1/15,\zeta_{3}]$, the vertical
  cycle $\mathcal{Z}_{\id{p}}$ is irreducible.
\end{coro}
\begin{proof}
  If $Z$ is a component of $S$ this follows directly from Theorem
  \ref{thm:3}. If $Z$ is a component of $F_{1}$ this follows from
  Theorem \ref{thm:1} and Theorem \ref{thm:2}.
\end{proof}

Let $\mathcal{D}'$ be the horizontal divisor of $\widetilde
{\mathcal{H}}$ determined by $D'$.

\begin{coro}
    For every prime $\id{p}$ of $\ZZ[1/15,\zeta_{3}]$, the divisor
    $\mathcal{D}'_{\id{p}}$ of the surface
    $\widetilde{\mathcal{H}}_{\id{p}}$ over $\Spec
    \ZZ[1/15,\zeta_{3}]/\id{p}$ is nef. 
\end{coro}
\begin{proof}
  Since the divisor $\mathcal{D}'_{\id{p}}$ is effective, we only have
  to show that the intersection of $\mathcal{D}'_{\id{p}}$ with any of
  its irreducible components is greater or equal to zero. By Corollary 
  \ref{cor:irredfib}, every irreducible component of
  $\mathcal{D}'_{\id{p}}$ is the specialization of a irreducible
  component of $D'$. Thus the result follows from the fact that $D'$
  is nef and that the intersection product is preserved by
  specialization.  
\end{proof}

\begin{coro}
Let $G$ be a Hilbert modular form of weight $(k_1,k_2)$, with
$k_1\equiv k_2 \pmod 2$, for $\Gamma_0(3)$ whose coefficients generate
a finite field extension $L$ of $\Q$. Let $\id{p}$ be a prime ideal of
$\Om_{L.F[\xi_3]}$ not dividing $15$. If $G$ vanishes with order $a$
at all cusps of ${\bar{\mathcal{H}}}_{\id{p}}$ and if the Fourier
coefficients of its $q$-expansion at the infinity cusp are algebraic
integers satisfying $a_\xi \equiv 0 \pmod{\id{p}}$ for all $\xi$ with
$\trace(\xi) \le (k_1+k_2)-3a$ then all $a_\xi$ are divisible by
$\id{p}$.
\label{coro:localsturmbound}
\end{coro}

\begin{proof} 
From the last Corollary, we have that the divisor
$\mathcal{D}'_{\id{p}}$ of the surface
$\widetilde{\mathcal{H}}_{\id{p}}$ over $\Spec
\ZZ[1/15,\zeta_{3}]/\id{p}$ is nef, so we can apply the same argument
as in the proof of Theorems \ref{thm:sturm1} and \ref{thm:Sturm2}. The
result follows from the fact that $\mathcal{D}'_{\id{p}}$ is the
specialization of the divisor $D$ consider in such theorems and the
fact that intersection numbers are preserved by specialization.
\end{proof}

\begin{remark}
In practice, if one starts with a form whose coefficients vanish at
all cusps of the complex Hilbert surface with order at least $a$ (for
example if it is a product of cusp forms, as will be the case in the
next section) then it also vanishes at all the cusps of
$\bar{\mathcal{H}}_{\id{p}}$ with order at least $a$.
\end{remark}

\subsection{The case of the form $(\overline{\ff\otimes \chi_{\sqrt{5}}})^3h_2^2 +
(\overline{\tau(\ff\otimes \chi_{\sqrt{5}}})^3h_1^2$ of level $\Gamma_0(6\sqrt{5})$}

We want to apply the results of the last two sections to the Hilbert
cusp form $\hh =(\overline{\ff\otimes \chi_{\sqrt{5}}})^3h_2^2 +
(\overline{\tau(\ff\otimes \chi_{\sqrt{5}}})^3h_1^2$ of Section $2$,
which has weight $(10,10)$ and level $\Gamma_0(6\sqrt{5})$. Assuming
that the $q$-expansion at the infinity cusp of $\ff$ is zero for all
elements of trace smaller than $\tilde{b}+1$ (hence the order of
vanishing at the four lines $L_i$, $i=1,\ldots,3$ is $3 \tilde{b}+3$),
we want to determine which value of $\tilde{b}$ forces the form $\hh$
(which vanishes with order $3$ at all cusps and $9\tilde{b}+9$ at the
four lines) to be the zero form. We start with some general results.




\begin{lemma} 
Let $\id{p},\id{q}$ be two distinct  prime ideals of $F$ relatively prime to $3$.
Then the index of $\Gamma_0(3\id{p}^r\id{q}^s)$ in $\Gamma_0(3)$ is
\[
[\Gamma_0(3):\Gamma_0(3\id{p}^r\id{q}^s)]=\normid{p}^{r-1}(\normid{p}+1)\normid{q}^{s-1}(\normid{q}+1).
\]
\label{lemma:subgroupindex}
\end{lemma}

In particular, 
\begin{equation}
[\Gamma_0(3):\Gamma_0(6\sqrt{5})]=30.
\end{equation}

Let $\hh$ be a modular form of weight $(10,10)$ for $\Gamma_0(6\sqrt{5})$,
which vanishes with order $3$ at all cusps, and with order $\tilde{b}+3$ at the
infinity cusp. We consider its norm to $\Gamma_0(3)$,
\[
G:=\prod_{\gamma \in \; \Gamma_0(6\sqrt{5})\backslash \Gamma_0(3)} \hh|_2[\gamma].
\]
It is a parallel weight $10\cdot 30$ Hilbert modular form for
$\Gamma_0(3)$. Since $\hh$ is a cusp form, looking at its q-expansion
at the different cusps it is easy to see that $G$ vanishes with order
at least $3\cdot 4$ at all cusps and with order at least
$3\cdot(3\tilde{b}+18)$ at the infinity cusp. Then
Theorem~\ref{thm:sturm1} and Corollary~\ref{coro:localsturmbound}
(with $a=12$, $b=9\tilde{b}+54$ and $k=150$) imply that if $\tilde{b}
\ge 181$ then $G$ is the zero form. So we get

\begin{thm} 
Let $\hh$ be a Hilbert modular form over $\overline{\FF}_2$ of
parallel weight $(10,10)$ for $\Gamma_0(6\sqrt{5})$, which vanishes
with order $3$ at all cusps. If its Fourier expansion is given by
\[
\hh= \sum_{\substack{\xi \gg 0\\ \xi \in \Om_F^\vee}} a_\xi \exp(\xi
z_1 + \tau(\xi) z_2). 
\]
with $a_\xi =0$ for all $\xi$ with $\trace(\xi) \le 181$, then $\hh=0$.
\label{thm:Sturm}
\end{thm}

\bibliographystyle{alpha}
\bibliography{al}

\Addresses
\end{document}